\documentclass[11pt,oneside]{amsart}

\usepackage[left=1in,right=1in,top=1in,bottom=1in,marginparwidth=0.8in]{geometry}
\usepackage{microtype}  

\usepackage{mathtools}
\usepackage{bbm}  

\swapnumbers

\usepackage{hyperref}

\usepackage{amssymb,amsfonts,amsmath,amsthm}
\usepackage[mathscr]{eucal}
\usepackage[alphabetic]{amsrefs}
\usepackage{stmaryrd}
\usepackage{colonequals}

\usepackage[ps,matrix,arrow,curve,cmtip]{xy}

\title{The Witt filtration of  Lubin-Tate deformation rings}
\author{Charles Rezk}
\date{ \today}
\address{Department of Mathematics \\
University of Illinois Urbana-Champaign \\ 
Urbana, IL}
\email{rezk@illinois.edu}

\numberwithin{equation}{section}

\makeatletter
  \let\c@subsection\c@equation
\makeatother

\theoremstyle{plain}   %% This is the default, anyway

\newtheorem{thm}[subsection]{Theorem}
\newtheorem*{thm*}{Theorem}
\newtheorem{prop}[subsection]{Proposition}
\newtheorem*{prop*}{Proposition}
\newtheorem{cor}[subsection]{Corollary}
\newtheorem*{cor*}{Corollary}
\newtheorem{lemma}[subsection]{Lemma}
\newtheorem*{lemma*}{Lemma}

\newtheorem*{claim*}{Claim}

\theoremstyle{remark}
\newtheorem{rem}[subsection]{Remark}    
\newtheorem*{rem*}{Remark}
\newtheorem{warn}[subsection]{Warning}
\newtheorem*{warn*}{Warning}
\newtheorem{exam}[subsection]{Example}
\newtheorem*{exam*}{Example}
\newtheorem*{ques*}{Question}

\theoremstyle{plain}

\begin{document}

% Note: use "\newcommand*" for macros with arguments whose arguments
% are not "long", i.e., do not contain blank lines or \par.

% For temporary questions.  For example, \margnote{This is something
% I'm confused about.} puts that message in the margin.
\newcommand{\margnote}[1]{\mbox{}\marginpar{\tiny\hspace{0pt}#1}}

\def\lambada{\lambda}

%%% standard operators for mathematics

% general categorical things
\newcommand{\id}{\operatorname{id}}
\newcommand{\colim}{\operatorname{colim}}
\newcommand{\llim}{\operatorname{lim}}
\newcommand{\Cok}{\operatorname{Cok}}
\newcommand{\Ker}{\operatorname{Ker}}
\newcommand{\Image}{\operatorname{Im}}
\newcommand{\op}{{\operatorname{op}}}
\newcommand{\Aut}{{\operatorname{Aut}}}
\newcommand{\End}{{\operatorname{End}}}
\newcommand{\Hom}{{\operatorname{Hom}}}

% shortcuts for arrows
\newcommand*{\ra}{\rightarrow}
\newcommand*{\lra}{\longrightarrow}
\newcommand*{\xra}{\xrightarrow}
\newcommand*{\la}{\leftarrow}
\newcommand*{\lla}{\longleftarrow}
\newcommand*{\xla}{\xleftarrow}

% homotopy theory
\newcommand{\ho}{\operatorname{ho}}
\newcommand{\hocolim}{\operatorname{hocolim}}
\newcommand{\holim}{\operatorname{holim}}

% macros for standard mathematical notations
\newcommand*{\realiz}[1]{\left\lvert#1\right\rvert}
\newcommand*{\len}[1]{\left\lvert#1\right\rvert}
\newcommand*{\Len}[1]{\left\lVert#1\right\rVert}
\newcommand{\set}[2]{{\{\,#1\,\mid\,#2\,\}}}
\newcommand{\bigset}[2]{\left\{\,#1\;\middle|\;#2\,\right\}}
\newcommand*{\tensor}[1]{\underset{#1}{\otimes}}
\newcommand*{\pullback}[1]{\underset{#1}{\times}}
\newcommand*{\powser}[1]{[\![#1]\!]}
\newcommand*{\laurser}[1]{(\!(#1)\!)}
\newcommand{\ndiv}{\not|}
\newcommand{\pairing}[2]{\langle#1,#2\rangle}

\newcommand{\lrtensor}[3]{{\mathstrut}^{#1}\!\otimes_{#2}\!{\mathstrut}^{#3}}
\newcommand{\ltensor}[2]{{\mathstrut}^{#1}\!\otimes_{#2}}
\newcommand{\rtensor}[2]{\otimes_{#1}\!{\mathstrut}^{#2}}

% some standard rings and fields
\newcommand{\F}{\mathbb{F}}
\newcommand{\Z}{\mathbb{Z}}
\newcommand{\N}{\mathbb{N}}
\newcommand{\R}{\mathbb{R}}
\newcommand{\Q}{\mathbb{Q}}
\newcommand{\C}{\mathbb{C}}

% topology
\newcommand{\point}{{\operatorname{pt}}}
\newcommand{\Map}{\operatorname{Map}}
\newcommand{\eev}{\wedge}
\newcommand{\sm}{\wedge} % smash product

% bb bold one
\newcommand{\bbone}{\mathbbm{1}}

% typesetting
\newcommand*{\mc}{\mathcal}
\newcommand*{\msc}{\mathscr}
\newcommand*{\mf}{\mathfrak}
\newcommand*{\mr}{\mathrm}
\newcommand*{\mb}{\mathbb}
\newcommand*{\ul}{\underline}
\newcommand*{\ol}{\overline}
\newcommand*{\wt}{\widetilde}
\newcommand*{\wh}{\widehat}
\newcommand*{\mtt}{\mathtt}
\newcommand*{\ms}{\mathsf}
\newcommand*{\mbf}{\mathbf}

% for defined words
\newcommand{\dfn}{\textbf}

% a ``backwards'' colon
\def\noloc{\;{:}\,}

% defining equals
\newcommand{\defeq}{\colonequals}
%\def\defeq{\overset{\mathrm{def}}=}

% to force a paragraph break at the start of theorems and proofs
\newcommand{\forcepar}{\mbox{}\par}

\renewcommand{\phi}{\varphi}

\newcommand{\Ab}{\mr{Ab}}

\renewcommand{\Im}{\operatorname{Im}}

\newcommand{\Res}{\operatorname{Res}}
\newcommand{\Tr}{\operatorname{Tr}}
\newcommand{\Ind}{\operatorname{Ind}}
\newcommand{\trJ}{J^{\Tr}}

\newcommand{\Spec}{\operatorname{Spec}}

\newcommand{\Set}{\mr{Set}}
\newcommand{\CAlg}{\mr{CAlg}}
\newcommand{\Sp}{\mr{Sp}}
\newcommand{\Mod}{\mr{Mod}}

\newcommand{\Alg}{\mr{Alg}}
\newcommand{\T}{\mathbb{T}}

\newcommand{\Sym}{\operatorname{Sym}}

\renewcommand{\O}{\mathcal{O}}
\newcommand{\m}{\mathfrak{m}}

\newcommand{\Witt}{\mathbb{W}}
\newcommand{\Wideal}{\mb{I}}
\newcommand{\gh}{\operatorname{gh}}

\newcommand{\Fr}{\mr{Fr}}

\newcommand{\Cent}{\operatorname{Cent}}
\newcommand{\AH}{\mr{AH}}

\newcommand{\length}{\operatorname{length}}

\newcommand{\Mark}{\operatorname{Mark}}
\newcommand{\oA}{\mathring{A}}

\newcommand{\Sip}{\Sigma^\infty_+}

\newcommand{\tr}{\operatorname{tr}}

\newcommand{\cnj}{\mr{conj}}

%%% abstract
\begin{abstract}
This note is a meditation on a generalization $\Witt_E$ of the classical
p-typical Witt vectors $\Witt_p$, which arises (geometrically) from
isogenies of 
deformations of formal groups, or (topologically) from the theory of
power operations on Morava $E$-theory.   For formal
groups of height 
$1$  we have $\Witt_E=\Witt_p$, but the $\Witt_E$ are richer when
height is $\geq 2$. We show that $\Witt_p$ splits naturally from
$\Witt_E$.  

A key property of $\Witt_E$ is the isomorphism 
$\pi_0E\approx \Witt_E(\pi_0E/\m)$, the ``cofreeness of the Morava
$E$-theory'' proved by Burklund, Schlank, and Yuan.  This isomorphism
determines a natural ``Witt filtration'' on $\pi_0 E$.   We describe how
this Witt filtration interpolates between the $p$-adic filtration and a
geometric filtration on $\pi_0E/(p)$.  We use this to give a new proof
of cofreeness.
\end{abstract}

%%% the title
\maketitle

%%% table of contents
% \tableofcontents

\section{Introduction}

In \cite{lubin-tate-formal-moduli} Lubin and Tate identified the space
of deformations of 1-dimensional commutative formal groups of  finite
height $h\geq 1$: given such a formal group $\Gamma$ over a perfect
field $\kappa$, its \dfn{universal deformation} is defined over a ring
$\O$ isomorphic to a power series ring
$\Witt_p\kappa\powser{u_1,\dots,u_{h-1}}$.  Morava observed
\cite{morava-noetherian-loc-cobordism} that there is a complex
orientable cohomology theory associated to any such universal
deformation, now called \dfn{Morava $E$-theory} or \dfn{Lubin-Tate
  theory}, with the property that $\pi_*E\approx \O[\mu^{\pm}]$ with
$\len{\mu}=2$.  Goerss, Hopkins, and Miller
\cite{goerss-hopkins-moduli-spaces} proved that the representing
spectrum of any Morava $E$-theory admits an essentially unique
commutative ring structure, and that it functorial in formal groups.
In particular, $E$ admits an action by the automorphism group of
$\Gamma/\kappa$.  

This commutative ring structure on $E$ endows it (and its cohomology theory)
with additional structure, namely \emph{power operations}, whose
structure has been explored in depth 
\cite{strickland-finite-subgroups-of-formal-groups},
\cite{strickland-morava-e-theory-of-symmetric}, 
\cite{ando-hopkins-strickland-h-infinity},
\cite{rezk-congruence-condition}. 
In particular, as shown in \cite{rezk-congruence-condition}, if $R\in
\CAlg_E(\Sp_{K(n)})$, then $\pi_0R \in \Alg_{\T}$, where the latter
denotes the category of algebras over a certain monad $\mathbb{T}$ on
the category 
$\CAlg_\O$ of (ordinary) commutative 
$\O$-algebras.
As noted in \cite{rezk-congruence-condition}*{4.23}, the forgetful
functor $\Alg_\T\ra \CAlg_\O$ is both \emph{monadic} and  \emph{comonadic}.

The goal of this  note is to explore  the underlying functor of this
comonad, which we 
denote $\Witt_E\colon \CAlg_\O\ra \CAlg_\O$ and call the
\dfn{Lubin-Tate-Witt vectors}, or more colloquially the
\dfn{$E$-typical Witt vectors}, to indicate its dependence on a choice
of Morava $E$-theory.   The name is chosen to reflect a connection to
the classical $p$-typical Witt 
vectors.  It was introduced in
\cite{burklund-schlank-yuan-chromatic-nullstellensatz}, where it is
denoted $\Witt_\T$. 

\begin{exam}[Lubin-Tate-Witt vectors at height 1]
Let  $\Gamma$ be a height 1 formal group law over $\kappa$, and $E$ its
associated Morava $E$-theory spectrum.  Then $\O=\Witt_p\kappa$, and
there is an isomorphism of functors $\Witt_E\approx \Witt_p$, where
$\Witt_p$ denotes the classical $p$-typical Witt vectors (restricted
to $\CAlg_{\Witt_p\kappa}$).  
Furthermore, the comonad structure on $\Witt_E$ recovers the theory of
\dfn{$\delta$-rings} ($\O$-algebras equipped with a
$p$-derivation) \cite{joyal-delta-rings}. 
\end{exam}

A striking feature of the $E$-typical Witt vectors is the following
result.

\begin{thm}[Cofreeness of the Lubin-Tate ring
  \cite{burklund-schlank-yuan-chromatic-nullstellensatz}*{Thm.\ 3.4}]
  \label{thm:cofreeness-of-O}
  The $\O$-algebra structure map $\O\ra 
  \Witt_E(\kappa)$ is an isomorphism.
\end{thm}
This was was proved by Burklund, Schlank, and Yuan in 2022 as one
component of their proof of the ``chromatic Nullstellensatz''.  The
author also gave a proof of this result (announced in 2019, but not
published), which will appear in a companion paper to this
\cite{rezk-cofreeness}.  Another proof was recently announced by Akhil
Mathew, using the theory of animated $p$-derivations.  We will give
yet another proof of cofreeness in this paper
(\S\ref{sec:witt-filt-lt-cofreeness}), different from our first but
similar in spirit to that of
\cite{burklund-schlank-yuan-chromatic-nullstellensatz}.  
  
``Cofreeness'' has interesting implications for the Lubin-Tate ring
and Morava $E$-theory.  The $E$-typical Witt vector
functor admits a canonical descending filtration $\Witt_E\supseteq
\Wideal_1\supseteq \Wideal_2\supseteq \cdots$ by ideals, so that
$\Witt_E\approx \llim\Witt_E/\Wideal_d$ and $\bigcap\Wideal_d=0$.
Cofreeness implies that this filtration is inherited by the Lubin-Tate
ring $\O$, and we refer to this (following \cite{burklund-schlank-yuan-chromatic-nullstellensatz}) as the \dfn{Witt filtration} of $\O$. 
We emphasize that the Witt filtration is \emph{functorial} in the formal group: in particular it is 
compatible with the action by $\Aut(\Gamma/\kappa)$ on $\O$.

For formal groups of height $h=1$, this is just $p$-adic filtration of
the $p$-typical Witt vectors of a perfect field:
$\Wideal_d=p^d\Witt_p\kappa$.  
However, when
$h>1$ the     
Witt filtration on $\O$ is  more subtle.

In fact, we will prove
the following (as \eqref{prop:witt-filtration-restriction},
\eqref{thm:witt-filtration-compat-with-p} and 
\eqref{prop:witt-trivial-quotient-filtration}, and using cofreeness). 

\begin{thm*}
  The Witt filtration $\{\Wideal_d\}$ of the Lubin-Tate ring $\O$ satisfies
  \begin{itemize}
  \item $\Wideal_d\cap \Witt_p\kappa= p^d\Witt_p\kappa$,
  \item $p\Wideal_d = \Wideal_{d+1}\cap p\O$,
  \item $\O/(p,\Wideal_{d+1})$ 
carries the universal example of a
    deformation of $\Gamma/\kappa$ to a ring of characteristic $p$
    which is isomorphic to $\Gamma^{(p^d)}/H$ for some finite subgroup
    scheme $H\leq 
    \Gamma^{(p^d)}$ of rank $p^d$. 
  \end{itemize}
\end{thm*}
That is, the Witt filtration is an $\Aut(\Gamma/\kappa)$-equivariant
filtration which interpolates between the $p$-adic
filtration on $\O$, and an evident on $\O/p$,
defined using the 
fact that $\O/p$ is a $\kappa$-algebra and so carries a 
trivial deformation of $\Gamma$.  However, the following question
remains open.
\begin{ques*}
  Is there a purely geometric characterization of the Witt filtration
  of $\O$?
\end{ques*}

An interesting consequence of all this is that we can use these ideas to
give a new proof of cofreeness of the Lubin-Tate ring, which is
largely independent of the one given by Burklund, Schlank, and
Yuan\footnote{Although our proof will refer to some lemmas in
  \cite{burklund-schlank-yuan-chromatic-nullstellensatz}, it
  completely avoids the more delicate aspects of their proof: namely,
  the  transchromatic induction of \S3.3 and its
   application in \S3.5 using the non-additive $p$-derivation.
   Otherwise, the ideas of the argument are  comparable: for
   instance, our \eqref{thm:witt-filtration-compat-with-p}(2) is 
   essentially 
   \cite{burklund-schlank-yuan-chromatic-nullstellensatz}*{3.45(1)}.},
and also very different from the earlier proof given by
  the author.  We give our new proof in
  \S\ref{sec:witt-filt-lt-cofreeness}.  

  In \cite{bssw-rational-kn}*{\S2.5}, the authors show that the
  $p$-typical Witt vectors are \emph{functorially} a summand of the Lubin-Tate
  ring, i.e., give an $\Aut(\Gamma/\kappa)$-equivariant isomorphism
  $\O\approx 
  \Witt_p\kappa\oplus \O'$  of abelian groups.  In this paper, we lift
  their splitting to the $E$-typical Witt vectors (proved as
  \eqref{thm:witt-splitting} below).
  \begin{thm*}
    There exist natural transformations
    $\Witt_p\xra{\beta_p} \Witt_E\xra{\sigma_p} \Witt_p$, functorial in $E$,
    such that (i) $\beta_p$ is a natural map of rings, (ii) $\sigma_p$ is a
    natural map of $\Witt_p$-modules, and (iii)
    $\sigma_p\beta_p \colon \Witt_p\ra \Witt_p$ is an automorphism,
    given by multiplication by a distinguished element $\AH^h_p\in
    \Witt_p(\Z_{(p)})^\times$.  
  \end{thm*}
Evaluation at $\kappa$ recovers the splitting of $\O$ given in
\cite{bssw-rational-kn}.

\begin{rem}
  The composite $\sigma_p\beta_p\colon \Witt_p\kappa\ra \Witt_p\kappa$ is
  precisely the map $f$ of \cite{bssw-rational-kn}*{Rem.\ 2.5.8}.  We
  are able answer the implicit question in their remark: $f$ is
  multiplication by $[(1-p)\cdots (1-p^{h-1})]^{-1}$
  \eqref{prop:where-does-one-go}.   
\end{rem}

\subsection{Structure of the paper}

In 
\S\ref{sec:e-typical-witt-vectors}--\S\ref{sec:ghost-components} we give an
introduction to the Lubin-Tate-Witt vectors and their basic
properties, including their ghost components.  In
\S\ref{sec:ghost-tower} we show 
that for perfect $\kappa$-algebras $A$ the ghost components assemble
to give a map $\Witt_E(A)/p\ra \llim \O_d\otimes_\O \kappa$, which is shown to be an isomorphism in \S\ref{sec:witt-filt-lt-cofreeness}.  In
\S\ref{sec:mult-p-witt-filtration}   we prove our main result about
the  interaction of the Witt
filtration of $\Witt_E$ multiplication by $p$ for perfect
$\kappa$-algebras, and then apply 
this in \S\ref{sec:witt-filt-lt-cofreeness} to discuss the image of
the Witt filtration in characteristic $p$, and to prove the cofreeness
of the Lubin-Tate ring. 

In \S\ref{sec:classical-witt} we review standard facts about the
classical Witt vectors. 
We apply this in \S\ref{sec:beta}--\S\ref{sec:splitting} to construct
the splitting of 
$\Witt_p$ from $\Witt_E$, by first constructing natural maps
$\Witt_\Z\xra{\beta} \Witt_E\xra{\sigma} \Witt_\Z$ and then showing
they restrict to the desired splitting on the summand $\Witt_p$ of
$\Witt_\Z$.

\subsection{Conventions}
\label{subsec:conventions}

This paper takes place mainly in 1-categories.   Thus, notation such
as ``$\CAlg_\O$'' refers to the 1-category of commutative algebras
over an ordinary ring $\O$, except when context dictates otherwise.

Given a fixed prime $p$, for  $m,h\geq1$ we write
\begin{align*}
  N^h_m &\defeq \len{\bigl\{\text{subgroups of order $m$ in
          $(\Q_p/\Z_p)^h$}\bigr\}},
\end{align*}
numbers which play a recurring role.

\subsection{Acknowledgements}

The author would like to thank William Balderrama and Nat Stapleton
for discussions related to work in this paper.

\section{Lubin-Tate-Witt vectors}
\label{sec:e-typical-witt-vectors}

Fix a formal group $\Gamma$ of height $0<h<\infty$ over a perfect
field $\kappa$. 
Let $E$ denote the corresponding Morava $E$-theory spectrum, and write
$\O=\pi_0E$ for the corresponding Lubin-Tate ring.

For any commutative $\O$-algebra $A$, we will define below an
$\O$-algebra $\Witt_E(A)$, which we refer to as the
\dfn{Lubin-Tate-Witt vectors} of $A$ (for the formal group associated to
$E$), or more informally as the ring of \dfn{$E$-typical-Witt vectors} of $A$.

Given  a commutative $\O$-algebra $A$, let $\Witt_E(A)$ denote the set of
sequences $a=(a_m)$ in $\prod_{m\geq0} E^0B\Sigma_m \otimes_{\O}
A$ such that
\begin{itemize}
\item $a_0=1$, and
\item $a_{i+j}|_{\Sigma_i\times \Sigma_j} = a_i\boxtimes a_j$ for all
  $i,j\geq0$.  
\end{itemize}
Here we write $a_{i+j}|_{\Sigma_i\times \Sigma_j}$ for the image of
$a_{i+j}\in E^0B\Sigma_{i+j}\otimes_{\O}A$ 
under the map induced by restriction along the inclusion
$\Sigma_i\times \Sigma_j\subseteq \Sigma_{i+j}$, while $a_i\boxtimes
a_j\in E^0(B\Sigma_i\times B\Sigma_j)\otimes_{\O} A$ is defined by
the K\"unneth map $(E^0B\Sigma_i\otimes_{\O}A)\otimes_A
(E^0B\Sigma_j\otimes_{\O}A)\ra E^0B\Sigma_i\times B\Sigma_j
\otimes_{\O}A$, which is an isomorphism.

We define two operations on the set $\Witt_E(A)$, so that for $a,b\in \Witt_E(A)$:
\begin{itemize}
\item $(a+b)_m \defeq \sum_{i+j=m} \Tr_{\Sigma_i\times
    \Sigma_j}^{\Sigma_m}(a_i\boxtimes b_j)$,
\item $(ab)_m \defeq a_mb_m$.
\end{itemize}
Here we write $\Tr_H^G$ for the transfer map associated to a subgroup
$H\leq G$.  We also define $1\in \Witt_E(A)$ by $1_m\defeq 1$ for all $m$, and
$0\in \Witt_E(A)$ by $0_1\defeq 1$ and $0_m\defeq 0$ for $m>0$.
\begin{prop}
  The above structure makes $\Witt_E(A)$ into a commutative ring.
\end{prop}
\begin{proof}
  First we show that $\Witt_E(A)$ is closed under these operations.  Closure
  under multiplication follows from the fact that the K\"unneth map is
  a ring homomorphism.  To prove closure under addition, let $S\approx
  \coprod_m B\Sigma_m$
  denote the groupoid of finite sets, which is symmetric 
  monoidal under disjoint union $\amalg$.  Then closure under addition follows
  from the push-pull formula for transfer applied to the commutative 
  square
  \[\xymatrix@C=100pt{
    {S\times S\times S\times S}
    \ar[r]^-{\tiny\begin{array}{l} (A,B,C,D) \\ \mapsto (A\amalg B,
      C\amalg D) \end{array}}
    \ar[d]_-{\tiny\begin{array}{l} (A,B,C,D) \\ \mapsto (A\amalg
      C,B\amalg D)\end{array}}
    & {S\times S} \ar[d]^-{\amalg}
    \\
    {S\times S} \ar[r]_-{\amalg}
    & {S}
}\]
which is seen to be a pullback of $\infty$-groupoids.

Unit, associativity, and commutativity of multiplication is imediate,
while those properties for addition are straightforward using
properties of  transfers.   The distributive law follows from the
push-pull formula applied to the pullback square
\[\xymatrix@C=75pt{
  {S\times S} \ar[r]^-{(A,B)\mapsto A\amalg B} \ar[d]_-{(A,B)\mapsto
    (A,B,A\amalg B)} 
  & {S} \ar[d]^{A\mapsto (A,A)}
  \\
  {S\times S\times S} \ar[r]_-{(A,B,C)\mapsto (A\amalg B, C)}
  & {S\times S}
  }\]
\end{proof}

The construction $\Witt_E$ defines a functor from
$\O$-algebras to commutative rings.  In fact, it takes values in
$\O$-algebras.  

Given $c\in \O=E^0(\mr{pt})$, let $P_m(c)\in E^0B\Sigma_m$ denote
the value of the power operation applied to $c$.  That is, given $c$,
represented as a map $f\colon \mb{S}\ra E$ of spectra, we let $P_m(c)$ be the
homotopy class of maps represented by the composite
\[
\Sigma^\infty_+B\Sigma_m\approx \mb{S}^{\otimes m}_{h\Sigma_m} \xra{f^{\otimes_m}_{h\Sigma_m}}
E^{\otimes m}_{h\Sigma_m} \xra{\mu_m} E,
\]
where $\mu_m$ comes from the $\mb{E}_\infty$-ring structure on $E$.

\begin{lemma}
  The operation $P(c)\defeq (P_m(c))$ defines a ring homomorphism
  $P\colon \O\ra \Witt_E(\O)$.  
\end{lemma}
\begin{proof}
  This is straightforward from standard properties of the power
  operation.  In particular, that $P$ takes values in $\Witt_E(\O)$
  amounts to
  \[
  P_{i+j}(c)|_{\Sigma_i\times \Sigma_j} = c_i\boxtimes c_j, \qquad P_0(c)=1,
\]
while that $P$ is a ring homomorphism amounts to
\[
P_m(c+d) = \sum_{i+j=m} \Tr_{\Sigma_i\times \Sigma_j}^{\Sigma_m}(
c_i\boxtimes d_j), \qquad P_m(cd)=P_m(c)P_m(d), \qquad P_m(1)=1.
\]
\end{proof}
Thus, every $\Witt_E(A)$ is canonically an $\O$-algebra via the
composite $\O \ra \Witt_E(\O)\ra \Witt_E(A)$, and so we have
constructed a endofunctor $\Witt_E\colon \CAlg_\O\ra \CAlg_\O$ on
the (ordinary) category of commutative $\O$-algebras.

\begin{rem} The above argument extends to show that for any $K(n)$-local
commutative $E$-algebra $R$,  the power operation 
supplies an $\O$-algebra map $P\colon \pi_0 R\ra \Witt_E(\pi_0R)$.
\end{rem}

\subsection{Multiplicative section}
\label{subsec:multiplicative-section}

We note that the evaluation map $a\mapsto a_1$ defines a natural
$\mc{O}$-algebra homomorphism  
$\pi\colon \Witt_E(A)\ra A$, and that $\pi$ admits a natural
multiplicative section $[\;]\colon A\ra \Witt_E(A)$, defined by
$[a]\defeq (a^m)$.

\section{Representability of $\Witt_E$}
\label{sec:representability-of-witt-e}

Next we recall that the functor $\Witt_E$ is corepresentable.  Write
\[
  T \defeq \bigoplus_{m\geq0} E_0^\vee B\Sigma_m,
\]
where $E_0^\vee X\defeq \pi_0\bigl[ E\otimes \Sigma^\infty_+
X\bigr]_{K(n)}$.   This is a commutative $\O$-algebra: the
multiplication is induced by 
transfers along the inclusions $\Sigma_i\times
\Sigma_j \leq \Sigma_{i+j}$.
\begin{prop}
  The functor $\Witt_E\colon \CAlg_\O\ra \Set$ is corepresented by
  $T$.  
\end{prop}
\begin{proof}
The $E_0^\vee B\Sigma_m$ are finitely generated free $\O$-modules
\cite{strickland-morava-e-theory-of-symmetric}*{Thm.\ 3.7},
and thus there is  an evident  bijection
\[
  \Hom_{\Mod_\O }(T, A) \approx \prod_m
  E^0B\Sigma_m\otimes_{\O} A. 
\]
It is straightforward to verify that
$\Hom_{\CAlg_\O}(T, A)$, as a subset of this, coincides
with $\Witt_E(A)$.
\end{proof}

As a consequence, elements $a=(a_m)\in \Witt_E(A)$ are determined by their
components $a_{p^d}$.

\begin{lemma}\label{lemma:structure-of-free-T-alg}
  The ring $T$ is a graded polynomial ring.  In particular, it has a
  homogeneous polynomial basis with $N_{m}^h$ generators  in degree
  $m$.  (The  $N_m^h$ are as in \S\ref{subsec:conventions}.)
\end{lemma}
\begin{proof}
  \cite{strickland-morava-e-theory-of-symmetric}*{\S5}.
\end{proof}

\begin{cor}\label{cor:pth-powers}
  If $a,b\in \Witt_E(A)$ are such that $a_{p^d}=b_{p^d}$ for all $d\leq e$,
  then $a_m=b_m$ for all $m<p^{e+1}$.  
\end{cor}
\begin{proof}
  Immediate from \eqref{lemma:structure-of-free-T-alg}.
\end{proof}

\begin{rem}
  Another way to think about the $E$-typical Witt vectors is that
  there are ring isomorphisms 
  \[
  \Witt_E(\pi_0 R) \approx \Hom_{\CAlg(h\Sp)}(\mathbb{S}\{x\}, R),
\]
natural in $R\in \CAlg(\Mod_E(\Sp_{K(n)}))$, where $\mathbb{S}\{x\}$
denotes the free commutative ring spectrum on one generator (and which
admits a cocommutative coalgebra structure).  
\end{rem}

\section{Witt filtration}

Our description of the generalized Witt ring gives rise to an evident
filtration by ideals (as noted in
\cite{burklund-schlank-yuan-chromatic-nullstellensatz}), called  the
\dfn{Witt filtration}. 
Explicitly, for $d\geq0$ define
\[
 \Wideal_d= \Wideal_d(A) \defeq \set{a\in \Witt_E(A)}{\text{$a_m=0$ if
     $p^d\nmid m$}}.    
\]
That is, $a\in \Wideal_d(A)$ when it is ``supported'' on $p^d\Z_{\geq 0}$.  
In view of \eqref{cor:pth-powers}, we can also say that
$a\in \Wideal_d(A)$ iff $a_m=0$ for $0<m<p^{d}$.

\begin{prop}\label{prop:witt-filtration-basic}
  The subset $\Wideal_d(A)\subseteq \Witt_E(A)$ is an ideal.  We have
  $\Wideal_0(A)=\Witt_E(A)$,
  $\Wideal_1(A)=\Ker[\Witt_E(A)\xra{a\mapsto a_1} A]$, and
  $\Wideal_d(A)\supseteq \Wideal_{d+1}(A)$.  The evident map
  \[
    \Witt_E(A) \xra{\sim} \lim_d \Witt_E(A)/\Wideal_d(A)
  \]
  is an isomorphism of $\O$-algebras, and $\bigcap_d \Wideal_d(A)=0$.
\end{prop}
\begin{proof}
  Straightforward using \eqref{lemma:structure-of-free-T-alg}.
\end{proof}

\begin{rem}
  The Witt filtration is \emph{not multiplicative}:  $\Wideal_d(A)
  \Wideal_e(A)$ is not generally
  contained in $\Wideal_{d+e}(A)$ (except when the height $h=1$).
\end{rem}

\subsection{The associated graded of the Witt filtration}

\begin{prop}[\cite{burklund-schlank-yuan-chromatic-nullstellensatz}*{Prop.\ 3.27}]\label{prop:witt-assoc-gr-abelian-gp}
  There is an isomorphism of abelian groups
  \[
  \upsilon_d \colon \Wideal_d(A)/\Wideal_{d+1}(A) \xra{\sim} A^{\oplus N^h_{p^d}},
\]
natural in commutative $\O$-algebras.
\end{prop}
\begin{proof}
  In view of \eqref{lemma:structure-of-free-T-alg}, we can construct a
  natural bijection by 
  evaluating at a set of polynomial generators in degree $p^d$ of the
  representing object $T$.  The group structure on the domain is
  induced by a comultiplication $T\ra T\otimes_\O T$ which preserves
  degree.
\end{proof}

\begin{prop}[\cite{burklund-schlank-yuan-chromatic-nullstellensatz}*{Prop.\
    3.29}]\label{prop:witt-assoc-gr-perfect}
  For perfect $\kappa$-algebras $A$ there is a natural isomorphism of
  $A$-modules 
  \[
   \Wideal_d(A)/\Wideal_{d+1}(A) \xra{\sim} A^{\oplus N^h_{p^d}}.
 \]
\end{prop}
\begin{proof}
  Recall the multiplicative section $c\mapsto [c]\colon A\ra
  \Witt_E(A)$ (\S\ref{subsec:multiplicative-section}).   Under the map
  of \eqref{prop:witt-assoc-gr-abelian-gp} we have
  \[
  \upsilon_d([c]a) = c^{p^d}a.
\]
Thus when $A$ is perfect we obtain an isomorphism of modules
$\phi^{-d}\circ \nu_d$ by composing with the $p^{-d}$-power map in
each component.
\end{proof}

\begin{prop}[\cite{burklund-schlank-yuan-chromatic-nullstellensatz}*{Prop.\
    3.30}]\label{prop:witt-quotient-basechange}
For perfect $\kappa$-algebras $A$ the evident ring homomorphism  
\[
\Witt_p(A)\otimes_{\Witt_p(\kappa)} \Witt_E(\kappa)\ra \Witt_E(A)
\]
induces isomorphisms
$\Witt_p(A)\otimes_{\Witt_p(\kappa)}\Witt_E(\kappa)/\Wideal_d(\kappa)
\xra{\sim} \Witt_E(A)/\Wideal_d(A)$.  
\end{prop}
\begin{proof}
  The ring homomorphism is induced by the universal property of
  $p$-typical Witt-vectors of perfect rings
  (\S\ref{subsec:p-typical-witt-of-perfect}).  The 
  isomorphisms are deduced by inductively applying 
  \eqref{prop:witt-assoc-gr-perfect}. 
\end{proof}

\section{Ghost components}
\label{sec:ghost-components}

For any $m\geq0$, the \dfn{transfer ideal} $\trJ_m$  in $E^0B\Sigma_m$ is
defined to be the image of
\[
(\Tr_{\Sigma_i \times \Sigma_{m-i}}^{\Sigma_m})\colon
\bigoplus_{0<i<m} E^0B\Sigma_i\times B\Sigma_{m-i}\ra E^0B\Sigma_m 
\]
the sum of the images of transfers from \emph{proper} subgroups of the
form $\Sigma_i\times \Sigma_{m-i}\leq \Sigma_m$.  It is in fact an
ideal, since $\Tr_H^G\colon E^0BH\ra E^0BG$ is an $E^0BG$-module map
(where $E^0BH$ is a $E^0BG$-module via the restriction map).
\begin{lemma}
  If $m\neq p^d$ then $\trJ_m$ is the unit ideal in  $E^0B\Sigma_m$.
  We have $\trJ_1=0$, while $\trJ_{p^d}$ for $d\geq1$ is 
  the
  image of the transfer map
  $\Tr^{\Sigma_{p^d}}_{(\Sigma_{p^{d-1}})^{\times p}}$.  
\end{lemma}
\begin{proof}
  If $m\neq p^d$, then there exists $0<i<m$ such that $\Sigma_i\times
  \Sigma_{m-i}$ contains a $p$-Sylow subgroup of $\Sigma_m$.  If
  $m=p^d$ then the $p$-Sylow subgroups of $\Sigma_i\times
  \Sigma_{m-i}$ for $0<i<m$ are conjugate to subgroups of
  $(\Sigma_{p^{d-1}})^p$.  
\end{proof}

Let 
\[
\O_d \defeq E^0B\Sigma_{p^d}/\trJ_{p^d}
\]
denote the corresponding quotient $\O$-algebra, called the \dfn{$d$th
  Strickland ring} and studied in 
\cite{strickland-finite-subgroups-of-formal-groups}, 
\cite{strickland-morava-e-theory-of-symmetric}.  We recall the following.
\begin{prop}[\cite{strickland-morava-e-theory-of-symmetric}*{Thm.\
    1.1}]\label{prop:o-d-free-rank}
  The ring $\O_d$ is a complete local ring.  
  As an $\O$-module, $\O_d$ is finitely generated and free of rank
  $N^h_{p^d}$.  
\end{prop}

For each $d\geq0$, the evident composite
\[
\Witt_E(A)\ra E^0B\Sigma_{p^d}\otimes_{\O}A \ra
(E^0B\Sigma_{p^d}/\trJ_{p^d})\otimes_{\O} A \approx \O_d\otimes_{\O}A
\]
is denoted $\gh_d$ and called the $d$th \dfn{ghost component}.  It is a ring
homomorphism
\[
\gh_d\colon \Witt_E(A) \ra \O_d\otimes_{\O}A.
\]
\begin{warn}
  The ghost components are not $\O$-algebra maps (except when $d=0$).  
\end{warn}

\begin{prop}
  The ideal $\Wideal_{d+1}(A)$ is contained in the kernel of
  $\gh_d$, whence $\gh_d$ factors through a ring homomorphism
  \[
  \ol\gh_d\colon \Witt_E(A)/\Wideal_{d+1}(A) \ra \O_d\otimes_{\O}A.
  \]
\end{prop}
\begin{proof}
  Immediate from the fact that $\gh_d(a)$ depends only on $a_{p^d}$.  
\end{proof}

\section{The ghost component tower for perfect $\kappa$-algebras}
\label{sec:ghost-tower}

The purpose of this section is to show that if $A$ is a perfect
$\kappa$-algebra, then the ghost components $\gh_d\colon \Witt_E(A)\ra
\O_d\otimes_\O A $ can be assembled into a ring homomorphism
\[
\gh_\infty\colon \Witt_E(A)/p \ra \lim_{\longleftarrow} \O_d\otimes_\O A.
\]
The target of $\gh_\infty$ is an inverse limit induced by 
ring  homomorphisms $\ol{\rho}_d\colon \O_d\otimes_\O \kappa \ra
 \O_{d-1}\otimes_\O\kappa$.  The
$\ol\rho_d$ are not $\kappa$-algebra maps, but rather satisfy
$\ol\rho_d(ca)= c^{1/p}\ol\rho_d(a)$ for $c\in \kappa$.  

\subsection{Base-change along Frobenius}

For a commutative ring $R$, we write $\phi\colon R\ra R/p$ for the
ring homomorphism defined by 
$\phi(x)=x^p$ (mod $p$).  Note that $R$ is a perfect $\F_p$-algebra if
and only if $\phi$ is an isomorphism.

We have two functors $\Mod_R\ra \Mod_R$:
\begin{align*}
  M^{(p)} &\defeq M\rtensor{R}{\phi} R/p,
  \\
  M^{[p]} &\defeq \Cok\bigl[ M^{\otimes p} \xra{N} (M^{\otimes
            p})^{\Sigma_p}\bigr]. 
\end{align*}
Here $M^{(p)}$ is regarded as an $R$-module via the right-hand factor
(and so in fact defines a functor $\Mod_R\ra \Mod_{R/p}$), while
$M^{[p]}$ is a quotient 
module of $(M^{\otimes p})^{\Sigma_p}$.   The map $N=\sum_{\sigma\in
  \Sigma_p} \sigma$, where the symmetric group permutes factors of
$M^{\otimes p}$.

Given $x\in M$  write $x^{(p)}\defeq x\otimes 1 \in M^{(p)}$.
The map $(-)^{(p)}\colon M\ra M^{(p)}$ is an abelian group homomorphism satisfying
$(cx)^{(p)}=c^px^{(p)}$ 
for $c\in R$.  Note that if $R$ is a perfect $\F_p$-algebra then
$(-)^{(p)}$ is always an isomorphism.

Similarly, we write 
$x^{[p]}\defeq [x^{\otimes p}]\in M^{[p]}$.
The function $(-)^{[p]}\colon M\ra M^{[p]}$ 
also satisfies $(cx)^{[p]}=c^px^{[p]}$, but is not a homomorphism in general.

\begin{lemma}
  Suppose $R$ is a $\Z_{(p)}$-algebra.  Then the map $(-)^{[p]}\colon
  M\ra M^{[p]}$ is a homomorphism of abelian groups, and $px^{[p]}=0$  
  for all $x\in M$.  Furthermore, it extends to a homomorphism of
  $R$-modules
  \[
  \gamma \colon M^{(p)} \ra M^{[p]}
\]
satisfying $\gamma(x^{(p)})=x^{[p]}$.  
\end{lemma}
\begin{proof}
  For $x,y\in M$, we have that
  \[
  N(x^{\otimes i} \otimes y^{\otimes p-i}) =
  i!\; (p-i)!\; \left[\text{symmetrization of $x^{\otimes i}\otimes
      y^{\otimes p-i}$}\right],
\]
Thus for $p$-local $R$, 
\[
(x+y)^{\otimes p} \equiv x^{\otimes p}+ y^{\otimes p} \mod \Im N,
\]
so $(-)^{[p]}$ is a homomorphism.  Similarly, $N(x^{\otimes p})=p!\;
x^{\otimes p}$ implies that $px^{\otimes p} \equiv 0\mod \Im N$.  
The map $\gamma$ is then defined  by $\gamma(x\otimes
a) \defeq a^p x^{[p]}$. 
\end{proof}

\begin{prop}
  If $R$ is a $\Z_{(p)}$-algebra and $M$ a flat $R$-module, then
  $\gamma\colon M^{(p)}\ra M^{[p]}$ is an isomorphism of $R$-modules.
\end{prop}
\begin{proof}
  The functors $M\mapsto M^{(p)}, M^{[p]}$ preserve filtered colimits,
  so we reduce to the case of finitely generated  free modules.  Given
  a choice of basis $e_1,\dots, e_n$ of $M$, a straightforward
  argument (using that $R$ is $p$-local) shows that $e_1^{[p]},\dots,
  e_n^{[p]}$ is a basis of $M^{[p]}$ as a free $R/p$-module.  The
  claim follows.
\end{proof}

In what follows, we will only work with flat modules over $p$-local
rings, so we freely make  use of the isomorphism
$M^{(p)}\approx M^{[p]}$.

It is convenient to note here that $(-)^{(p)}=(-)^{(p)_R}\colon
\Mod_R\ra \Mod_{R}$ is 
lax symmetric monoidal, and so sends commutative $R$-algebras to
commutative $R$-algebras.  Furthermore, its formation is compatible
with base change: we have $(M\otimes_R S)^{(p)_S}\approx
M^{(p)_R}\otimes_R S$ for ring homomorphisms  $R\ra S$.

\subsection{The maps $\rho_d$}

We now define $\O$-algebra homomorphisms $\rho_d\colon \O_d\ra
\O_{d-1}^{(p)}$.  In fact, we will construct $A$-algebra homomorphisms
\[
\rho_d\colon \O_d\otimes_\O A\ra (\O_{d-1}\otimes_\O A)^{(p)}=
\O_{d-1}^{(p)}\otimes_\O A
\]
which are natural in the $\O$-algebra $A$, so that specializing to
$A=\O$ gives the desired map.

\textit{Construction of $\rho_d$.}  Consider the restriction and
transfer maps
\[
\Res\colon
E^0B\Sigma_{p^d}\otimes_\O A\rightleftarrows E^0B\Sigma_{p^{d-1}}^p\otimes_\O A
\noloc\! \Tr.
\]
As the $E^0B\Sigma_m$ is are finitely generated free $\O$-modules, they are flat,
and they participate in K\"unneth isomorphisms.  Thus $\Res$ induces a
map of $A$-algebras
\[
E^0B\Sigma_{p^d}\otimes_\O A \ra \bigl[
E^0B\Sigma_{p^{d-1}}\otimes_\O A]^{[p]},
\]
since $\Res(a) \in (E^0(B\Sigma_{p^{d-1}}\otimes_\O A)^{\otimes
  p})^{\Sigma_p}\subseteq E^0(B\Sigma_{p^{d-1}}\otimes_\O A)^{\otimes p}\approx
E^0B\Sigma_{p^{d-1}}^p\otimes_\O A$.   
I claim this descends to quotients by transfer ideals, giving
\[
E^0B\Sigma_{p^d}/\trJ_{p^d}\otimes_\O A \ra
[E^0B\Sigma_{p^{d-1}}/\trJ_{p^{d-1}}\otimes_\O A]^{[p]},
\]
which is the desired map $\rho_d$.

We prove this by exploiting the functoriality of Borel equivariant
cohomology with respect to spans of finite covering maps (aka the
``double-coset formula''). 
Write $X=\Sigma_{p^d}/\Sigma_{p^{d-1}}^p$, which
we identify with the set of functions
$f\colon \ul{p^d}\ra \ul{p}$ whose fibers all have size $p^{d-1}$.
This decomposes into $\Sigma_{p^d}$-orbits as
\[
X\times X \approx \coprod_{\sigma\in \Sigma_p} X_\sigma \;\amalg\;
\coprod_\alpha Y_\alpha. 
\]
Here $X_\sigma=\set{(f,\sigma f)}{f\in X}\subseteq X\times X$ is an
orbit isomorphic to $X$, while the each remaining orbits $Y_\alpha$ 
has isotropy groups conjugate to subgroups of the form
$H_\alpha = \Sigma_{p^{d-1}}^{i-1}\times (\Sigma_i\times \Sigma_{p^{d-1}-i})
\times \Sigma_{p^{d-1}}^{p-i}$ with $0<i\leq p$.
From this we read off that
\[
\Res\Tr(x) = \sum_{\sigma\in \Sigma_p} \sigma(x) + \sum_{\alpha}
y_\alpha \in E^0B\Sigma_{p^{d-1}}^p\otimes_\O A,
\]
where $y_\alpha$ is in the image of the transfer from $H_\alpha$, and
so projects to $0$ in $(E^0B\Sigma_{p^{d-1}}/\trJ_{p^{d-1}}\otimes_\O
A)^{\otimes p}$.

\begin{prop}
The diagram 
\[\xymatrix@C=35pt{
  {\Witt_E(A)} \ar[r]^-{\gh_d} \ar[d]_{\gh_{d-1}}
  & {\O_d\otimes_\O A}  \ar[d]^{\rho_d}
  \\
  {\O_{d-1}\otimes_\O A} \ar[r]_--{x\mapsto x^{(p)}}
  & {(\O_{d-1}\otimes_\O A)^{(p)}} 
  }\] 
commutes.  
\end{prop}
\begin{proof}
  Immediate from the construction of $\rho_d$, since for $a\in
  \Witt_E(A)$ we have that $\Res(a_{p^d})= (a_{p^{d-1}})^{\otimes p}$. 
\end{proof}

\subsection{The maps $\ol{\rho}_d$}
\label{subsec:ol-rho-d}

When $A$ is a perfect $\kappa$--algebra, the map along the bottom is an
isomorphism.  In this case write $\ol\rho_d$ for the composite of
\[
\O_d\otimes_\O A \xra{\rho_d} (\O_{d-1}\otimes_\O A)^{(p)} \xla{\approx}
\O_{d-1}\otimes_\O A.
\]
This is isomorphic to a map of the form $\ol\rho_d\otimes \id_A$, where
$\ol\rho_d \colon \O_d\otimes_\O \kappa\ra \O_{d-1}\otimes_\O \kappa$
is  the 
specialization to $A=\kappa$.  Note that $\ol\rho_d$ is not an
$\kappa$-algebra map, but instead satisfies
$\ol\rho_d(cx)=c^{1/p}\ol\rho_d(x)$ for $c\in \kappa$.

\begin{cor}\label{cor:rho-compat-with-ghost}
  For perfect $\kappa$-algebras $A$ the diagram
  \[\xymatrix{
    {\Witt_E(A)} \ar[r]^{\gh_d} \ar[dr]_{\gh_{d-1}}
    & {\O_d\otimes_\O A} \ar[d]
    & {(\O_d\otimes_\O \kappa)\otimes_\kappa A} \ar[l]_-{\approx}
    \ar[d]^{\ol\rho_d\otimes \id}
    \\
    & {\O_{d-1}\otimes_\O A}
    & {(\O_{d-1}\otimes_\O\kappa)\otimes_\kappa A} \ar[l]_-{\approx}
  }\]
commutes.
\end{cor}

We write $\gh_\infty\colon \Witt_E(A)\ra \lim
\O_d\otimes_\O A$ for the map to the limit.  As the target is of
characteristic $p$, it factors through a map $\Witt_E(A)/p\ra \lim
\O_d\otimes_\O A$, which we also denote by $\gh_\infty$.

\section{Multiplication by $p$ in the Witt filtration}
\label{sec:mult-p-witt-filtration}

We prove our main result on the Witt filtration of $\Witt_E(A)$.

\begin{thm}\label{thm:witt-filtration-compat-with-p}
  Let $A$ be a perfect $\kappa$-algebra.  Then $\Witt_E(A)$ is
  $p$-torsion free.  Furthermore,  for every $d\geq0$ we have that
  $p\Wideal_d(A)= \Wideal_{d+1}(A)\cap p\Witt_E(A)$.
\end{thm}

\begin{cor}
  Under the same hypotheses, we have
  $p^e\Wideal_d(A)=\Wideal_{d+e}(A)\cap p^e\Witt_E(A)$ for all
  $e\geq0$.  
\end{cor}

Recall the $\O$-algebra $T$ which corepresents the functor $\Witt_E$.
As noted earlier (\S\ref{sec:representability-of-witt-e}),  $T\approx
\bigoplus T_m$ with 
$T_m=E_0^\eev B\Sigma_m$.   
Multiplication by $[p]$ in $\Witt_E$ is thus represented by a grading
preserving 
$\O$-algebra map $[p]\colon T\ra T$.

We only care right now about the restriction of $\Witt_E$ to
$\kappa=\O/\mf{m}$-algebras.  Thus we write
$\ol{T}_m=T_m\otimes_{\O}\kappa$ and $\ol{T}=\bigoplus
\ol{T}_m$, and $[p]\colon \ol{T}\ra \ol{T}$ for the induced map.

\begin{prop}\label{prop:image-of-pth-power-on-T-mod-kappa}
  The image of $[p]\colon \ol{T}\ra \ol{T}$ is precisely the
  subring $\ol{T}^p\subset \ol{T}$ of $p$th powers.  
\end{prop}

We will prove this Proposition below.  First, we use it to prove the Theorem.

\begin{proof}[Proof of \eqref{thm:witt-filtration-compat-with-p}]
Given an element
$a\in
\Witt_E(A)$ we also denote by $a\colon
\ol{T}\ra A$ its representing map.  Note that $a$ represents $0\in
\Witt_E(A)$ iff 
$a(\ol{T}_m)=0$ for all $m>0$.

First we show that $\Witt_E(A)$ is $p$-torsion free.  If $pa=0$, i.e.,
if $(a\circ [p])(\ol{T}_m)=0$ for $m>0$, then by
\eqref{prop:image-of-pth-power-on-T-mod-kappa} we 
have that $a((\ol{T}_m)^p)=0$.  But since $A$ is perfect and
$a(x)^p=a(x^p)$, it follows that $a(x)=0$ for all $x\in\ol{T}$ in
positive grading.

Next we show that $p\Wideal_d(A)=\Wideal_{d+1}(A)\cap p\Witt_E(A)$.
Recall that $a\in \Wideal_d(A)$ if and only if $a(\ol{T}_m)=0$
whenever $p^d\nmid m$.

By \eqref{prop:image-of-pth-power-on-T-mod-kappa} we see that for all
$m\geq0$ we have that $[p](\ol{T}_m) = (\ol{T}_{m/p})^p$, where
we write $\ol{T}_{m/p}=0$ when $m/p$ is not an integer.
Thus, for any $m$ we have that  $(pa)(\ol{T}_m)=0$ if and only if
$a((\ol{T}_{m/p})^p)=0$, and so (because $A$ is perfect) if and only
if 
$a(\ol{T}_{m/p})=0$.   The claim follows immediately.
\end{proof}

\begin{rem}
  When the height $h=1$ (e.g., the classical $p$-typical Witt vectors
  of  a perfect $\F_p$-algebra), we have that $\Wideal_1(A)\subseteq
  p\Witt_E(A)$, 
  whence $p\Wideal_d(A)=\Wideal_{d+1}(A)$ for all $d$, and thus
  $\Wideal_d(A)=p^d\Witt_E(A)$.  
\end{rem}

To prove the proposition, we look at the structure of $T$ as a
graded abelian Hopf algebra.  Its multiplication and comultiplication
maps are induced  by the evident map  $B\Sigma_i\times B\Sigma_j\ra
B\Sigma_{i+j}$ and its transfer $\Sip B\Sigma_{i+j}\ra \Sip
B\Sigma_i\otimes \Sip B\Sigma_j$ respectively.  As each
$T_m=E_0^\eev B\Sigma_m$ is a finitely generated free
$\O$-module, we obtain a dual Hopf algebra $T^\vee = \bigoplus
T_m^\vee = \bigoplus_m \Hom_\O(T_m, \O)$.

\begin{prop}\label{prop:T-is-self-dual}
  There is an isomorphism $T\approx T^\vee$ of graded Hopf
  algebras over $\O$.
\end{prop}
\begin{proof}
  This is an immediate consequence of ambidexterity, in particular of
  the 1-semiadditivity of $\Sp_{K(n)}$.  In particular, for finite
  groups  $G$ there  are natural $K(n)$-local equivalences $\Sip BG \approx
  \mc{F}(\Sip BG, \mb{S}_{K(n)})$, which interchange restriction and
  transfer maps.  Everything we need can be found in
  \cite{strickland-kn-local-duality-finite-groups}. 
\end{proof}

\begin{proof}[Proof of \eqref{prop:image-of-pth-power-on-T-mod-kappa}]
  From \eqref{prop:T-is-self-dual} we have that $\ol{T}\approx
  \ol{T}^\vee$, an isomorphism of  
  graded Hopf algebras over the perfect field  $\kappa$ (using the
  $\kappa$-linear 
  dual). We also know that $T$, and thus $\ol{T}$, is a polynomial
  ring with homogeneous generators \eqref{lemma:structure-of-free-T-alg}.

  As both $\ol{T}$ and $\ol{T}^\vee$ are both abelian Hopf
  algebras over a perfect field of 
  characteristic $p$, we can factor their ``$p$-th power
  endomorphisms'' $[p]$ as composites of $\kappa$-algebra homomorphisms
  \[
    \ol{T} \xra{V} \phi^*\ol{T} \xra{F} \ol{T}\qquad
    \text{and} \qquad
    \ol{T}^\vee \xra{V} \phi^*\ol{T}^\vee \xra{F} \ol{T}.
  \]
  Here $\phi^*$ denotes basechange along the $p$th power map
  $\phi\colon \kappa\ra \kappa$, and $F$ is the relative $p$-th power
  Frobenius.  Furthermore, duality interchanges $V$ and $F$.

  As $\ol{T}$ is a polynomial ring, $F\colon \phi^*\ol{T}\ra
  \ol{T}$ is injective, whence  $V=F^\vee\colon
  \ol{T}^\vee\ra 
  \phi^*\ol{T}^\vee$ is  surjective.  The isomorphism
  $\ol{T}\approx \ol{T}^\vee$ then implies that
  $V\colon \ol{T}\ra \phi^*\ol{T}$ is surjective, and therefore
  the image of $[p]\colon \ol{T}\ra \ol{T}$ is precisely the
  image of $F$, as desired.
\end{proof}

\section{Witt filtration of Lubin-Tate space and cofreeness}
\label{sec:witt-filt-lt-cofreeness}

There are two ring homomorphisms $s_d,t_d\colon \O\ra
\O_d$  of interest.  They are both defined as composites
\[
\xymatrix{
  {\O} \ar@{=}[r]
  & {E^0(\point)} \ar@<1ex>[r]^-{\pi^*} \ar@<-1ex>[r]_-{P_{p^d}}
  & {E^0B\Sigma_{p^d}} \ar@{->>}[r]
  & {E^0B\Sigma_{p^d}/\trJ_d} \ar@{=}[r]
  & {\O_d}
}
\]
using (for $s_d)$ restriction along $\pi\colon B\Sigma_{p^d}\ra
\point$ or (for $t_d$) 
the power operation $P_{p^d}$.  The map $s_d$ defines the ``usual''
$\O$-algebra structure on $\O_d$ (used as such up to now in this
paper), while $t_d$ is equal to the
composite of $\O\xra{P}\Witt_E(\O)\xra{\gh_d}\O_d$.

\begin{prop}\label{prop:t-surjective}
The map $t_d$ induces a surjective map $\ol{t}_d\colon \O/p\ra
\O_d\otimes_\O \kappa$.  
\end{prop}
\begin{proof}
  This is
  \cite{burklund-schlank-yuan-chromatic-nullstellensatz}*{Prop.\
    3.35}, a consequence of the fact that  $s_d$
  and $t_d$ induce a surjective   
  homomorphism $\O\wh{\otimes}_{\mathbb{W}_p} \O\ra \O_{d}$
  \cite{burklund-schlank-yuan-chromatic-nullstellensatz}*{Prop.\
    3.39}.
\end{proof}

Thus we have 
$(\O/p)/\ol{I}_{d+1}\xra{\sim} \O_d\otimes_\O \kappa$,  where we define
$\ol{I}_{d+1}\defeq \ker 
\ol{t_d}$.

Recall the maps $\ol\rho_d\colon \O_d\otimes_\O\kappa \ra
\O_{d-1}\otimes_\O \kappa$ of \S\ref{subsec:ol-rho-d}.
\begin{prop}\label{prop:rho-t-compat}
The maps $\ol\rho_d$ are surjective, and we have $\ol\rho_d\circ
\ol{t}_d= \ol{t}_{d-1}$.   
\end{prop}
\begin{proof}
  Naturality of ghost components means that $\ol{t}_d$ is equal to the
  composite of
  \[
  \O\xra{P} \Witt_E(\O) \xra{\Witt_E(\pi)} \Witt_E(\kappa) \xra{\gh_d}
  \O_d\otimes_\O \kappa.
\]
The identity $\ol{\rho}_d\circ \ol{t}_d=\ol{t}_{d-1}$ is then
immediate from \eqref{cor:rho-compat-with-ghost}.  Surjectivity of
$\ol\rho_d$ follows from that of $\ol{t}_d$ \eqref{prop:t-surjective}.  
\end{proof}

Write
\[ 
\ol{I}_{d+1} \defeq \ker \ol{t}_d,
\]
an ideal in $\O/p$.   As a consequence of \eqref{prop:t-surjective} and
\eqref{prop:rho-t-compat}, we obtain a descending chain of ideals
\[
\O=\ol{I}_0\supseteq \ol{I}_1\supseteq \ol{I}_2\supseteq \cdots.
\]

\subsection{The trivial-quotient filtration on Lubin-Tate space in
  characteristic $p$} 

Consider an Artinian local ring $A$ of characteristic $p$, whose
residue field $\kappa_A$ is perfect.  This implies that the projection
$A\ra \kappa_A$ has a unique section, and so we can canonically 
identify $\kappa_A$ with a subring of $A$.

Recall that a \dfn{deformation} of $\Gamma/\kappa$ to $A$ is data
$(G,i,f)$, where $G$ is a formal group over $A$, $i\colon \kappa\ra
\kappa_A$ is a ring homomorphism, and $f\colon i^*\Gamma\xra{\sim}
G_0$ is an isomorphism of formal groups over $\kappa_A$, where $G_0$
is the base change of $G$ to the residue field.  

Say that a deformation $(G,i,f)$ of $\Gamma$ to such an $A$ is in 
\dfn{trivial-quotient filtration $\leq d$} if there exists a ring
homomorphism $\wt{i}\colon \kappa\ra \kappa_A\subseteq A$ and an isogeny
$g\colon \wt{i}^*\Gamma\ra G$  of formal groups over $A$ of degree $p^d$,
such that $g_0=f\circ \Fr^d$, where $g_0$ denotes the base change of
$g$ to the residue field, and $\Fr^d$ is the relative $p^d$th power
Frobenius isogeny.  (This implies that $\wt{i}=i\circ \phi^{-d}$.)

\begin{rem}\label{rem:trivial-quotient-filtration-alt}
  We can rephrase this in the following way: $(G,i,f)$ is in
  trivial-quotient filtration $\leq d$ if and only if there exists a
  ring homomorphism $j\colon \kappa \ra \kappa_A\subseteq A$ and an
  isogeny $g\colon j^*\Gamma\ra G$ of formal groups over $A$ of degree
  $p^d$.  The point is that, after basechange to the residue field, we
  have that $g_0\circ h = f\circ \Fr^d$ for an isomorphism $h\colon
  (i\phi^{-d})^* \Gamma \xra{\approx} j^*\Gamma$ of formal groups
  defined over  $\kappa_A\subseteq A$, and so we can replace $j$ by
  $i\phi^{-d}$ and $g$ by $g\circ h$.  
\end{rem}

\begin{prop}\label{prop:witt-trivial-quotient-filtration}
  Let $(G,i,f)$ be a deformation of $\Gamma$ to such a ring $A$, and
  which is classified by a map $\phi\colon \O/p\ra A$.  Then 
  the following are equivalent. 
  \begin{enumerate}
  \item $\phi(\ol{I}_{d+1})=0$.
  \item $(G,i,f)$ is in trivial-quotient filtration $\leq d$.
  \end{enumerate}
\end{prop}
\begin{proof}
We recall the universal property of $\O_d$: it carries the universal
example of an isogeny $g\colon G'\ra G$ of degree $p^d$ between
deformations $(G',i',f'), (G,i,f)$ such that $g_0=\Fr^d$, and so that
$(G',i',f')$ and $(G,i,f)$ are classified by $s_d$ and $t_d$
respectively 
(\cite{strickland-finite-subgroups-of-formal-groups},
\cite{strickland-morava-e-theory-of-symmetric}, see also
\cite{rezk-congruence-condition}*{\S12}).

Since $t_d$ factors as $\O\twoheadrightarrow (\O/p)/\ol{I}_{d+1}\approx
\O_d\otimes_\O \kappa$, the claim is immediate.
\end{proof}

\begin{prop}\label{prop:iso-t-tower}
  The maps $\ol{t}_d$ assemble to give an isomorphism of rings $\ol{t}\colon
  \O/p\xra{\sim} \lim \O_d\otimes_\O \kappa$.
\end{prop}
\begin{proof}
  We have identified the inverse system $\{\O_d\otimes_\O\kappa\}$ 
  with $\{(\O/p)/\ol{I}_{d+1}\}$.  Write $\ol{\m}\subseteq \O/p \approx
  \kappa\powser{u_1,\dots, u_{h-1}}$ for the maximal ideal.

  The ring $(\O/p)/\ol{I}_{d+1}\approx \O_d\otimes_\O\kappa$ is
  artinian, and of finite length as a $\kappa$-module
  \eqref{prop:o-d-free-rank}.   Thus
  for any $d$ there exists $k$ such that $\ol\m^k \subseteq
  \ol{I}_{d+1}$.  It remains to show that for any $k$ there exists $d$
  such that $\ol{I}_{d+1}\subseteq \ol\m^k$.

  Thus, it suffices 
to show that for all factorizations $\kappa\xra{i}
  A\xra{\pi} \kappa$ of $\id_\kappa$ with $A$ an Artinian local ring,
  that for any $\kappa$-algebra map $\phi\colon \O/p\ra $ there exists
  $d\geq1$ such that $\phi(\ol{I}_{d+1})=0$.  (That is, take
  $A=(\O/p)/\ol\m^k$.)

  Recall that $\phi$ classifies a deformation $(G,\id,f)$ of $\Gamma/\kappa$
  to $A$, with $f\colon \Gamma\xra{\approx} G_0$ an isomorphism of
  formal groups over  $\kappa$.
  By the hypotheses on $A$, $p=0$ and $\ker{\pi}\subseteq A$ is a
  nilpotent ideal, and so by the argument of
\cite{drinfeld-coverings-of-p-adic-symmetric-domains}*{App.}, we can
lift the isomorphism $f$ to $A$ up to a  
  power of $p$.  That is, there exists $k\geq0$ and an isogeny of
  formal groups $g\colon i^*\Gamma\ra G$ over $A$ such that
  $g_0=f_0\circ p^k$ (so that $g$ has degree $p^{kh}$).

  In view of \eqref{rem:trivial-quotient-filtration-alt}, this means
  that $(G,\id,f)$ is in trivial-quotient filtration $\leq kh$, i.e.,
  $\phi(\ol{I}_{kh+1})=0$, as desired.
\end{proof}

\subsection{Proof of cofreeness}

\begin{prop}\label{prop:witt-quotient-p-exact-sequence} 
  For any perfect $\kappa$-algebra $A$ and $d\geq0$ we have a short exact sequence
  \[
    0 \ra \Witt_E(A)/\Wideal_d \xra{p} \Witt_E(A)/\Wideal_{d+1}
    \xra{\gh_d} \O_d\otimes_\O A \ra 0.
    \]
  \end{prop}
  \begin{proof}
    
    By \eqref{thm:witt-filtration-compat-with-p} the first map in the
    sequence is injective.  Thus we 
    must show that the induced map $\phi\colon
    \Witt_E(\kappa)/(p,\Wideal_{d+1})\ra 
    \O_d\otimes_\O A$ is an isomorphism.       Using
    \eqref{prop:witt-quotient-basechange}, we reduce to the 
    case of $A=\kappa$. 
 
    When $A=\kappa$, the composite $\O\ra
    \Witt_E(\kappa)/(p,\Wideal_{d+1})\xra{\phi} \O_d\otimes_\O \kappa$
    is surjective \eqref{prop:t-surjective} whence $\phi$ is
    surjective.  We also have that
    $\length_\kappa\Witt_E(\kappa)/\Wideal_d=N_1^h+\cdots + N^h_{p^{d-1}}$
    \eqref{prop:witt-assoc-gr-perfect}, 
    whence $\length_\kappa\Witt_E(\kappa)/(p,\Wideal_{d+1})=N^h_{p^d}$.
    Since also $\length_\kappa \O_d\otimes_\O \kappa =\dim_\kappa
    \O_d\otimes_\O \kappa= N^h_{p^d}$ \eqref{prop:o-d-free-rank}, the
    map $\phi$ is also injective. 
  \end{proof}

  \begin{cor}\label{cor:witt-mod-p-sequence}
    For any perfect $\kappa$-algebra $A$, we have a short exact sequence
    \[
      0\ra \Witt_E(A) \xra{p} \Witt_E(A) \xra{\gh_\infty} \llim
      \O_d\otimes_\O A\ra 0.
      \]
\end{cor}
\begin{proof}
  Take the limit of the exact sequences of
  \eqref{prop:witt-quotient-p-exact-sequence}, and the fact 
  that, for any $\O$-algebra $A$, we have $\Witt_E(A)\xra{\sim} \llim
  \Witt_E(A)/\Wideal_d$ and $\llim^1 \Witt_E(A)/\Wideal_d=0$.
\end{proof}

\begin{lemma}\label{lemma:ab-group-thing}
  Let $f\colon M\ra N$ be a homomorphism of abelian groups which
  induces an isomorphism $M/p\xra{\sim} N/p$.  If $M\xra{\sim}
  \llim M/p^k$ and $\bigcap p^kM=0$, and $N$ is
  $p$-torsion free and $\bigcap p^kN=0$, then $f$ is an isomorphism.
\end{lemma}
\begin{proof}
  Straightforward.
\end{proof}

We can now deduce the cofreeness of the Lubin-Tate ring.
\begin{proof}[Proof of \eqref{thm:cofreeness-of-O}]
  Apply the lemma \eqref{lemma:ab-group-thing} to the map $\O\ra
  \Witt_E(\kappa)$, using \eqref{prop:iso-t-tower} and
  \eqref{cor:witt-mod-p-sequence}.  The hypothesis 
  that $\bigcap p^k\Witt_E(\kappa)=0$ is
  \eqref{prop:witt-filtration-basic}. 
\end{proof}

\subsection{The Witt filtration of the Lubin-Tate ring at small
  heights}

Given that $\O\approx \Witt_E(\kappa)$, we write $I_d\subseteq \O$ for
the Witt ideal corresponding to $\Wideal_d(\kappa)\subseteq \Witt_E(\kappa)$.  

\begin{exam}[Height 1]
  When $h=1$, we have that $\O=\Witt_p(\kappa)$, and in this case the
  Witt filtration is just the $p$-adic filtration: $I_d=p^d\O$.  Note
  that $\O/(p,I_{d+1})\approx \O_d\otimes_\O \kappa\approx  \kappa$ for all
  $d\geq 0$. 
\end{exam}

When the height $h=2$, we have that $\O\approx
\Witt_p\kappa\powser{a}$, and furthermore $\O/(p,I_{d+1})\approx
\O_d\otimes_\O\kappa \approx \kappa\powser{a}/(a^{n_d})$ where
$n_d=N_{p^d}^2=1+p+\cdots+ p^{d-1}$ (as in
\S\ref{subsec:conventions}). 
\begin{prop}
Fix a formal group of height 2 over $\kappa$, and identify $\O\approx
\Witt_p\kappa\powser{a}$.  Then there exist polynomials
$f_d\in \Witt_p\kappa [a]$ for $d\geq0$ such that
\begin{enumerate}
\item $f_1=a$,
\item $f_d$ is monic of degree $n_d$, 
\item $f_d \equiv a^{n_d} \pmod{p}$, 
\item $O/I_d \approx \O/(p^d,  p^{d-1}f_1,\dots, f_{d-1})$ for
  all $d\geq0$.  
\end{enumerate}
\end{prop}
\begin{proof}
Since $\O/(p,I_{d+1})\approx \kappa[a]/(a^{n_d})$, we can choose a monic
polynomial $f_d\in 
  I_{d+1}$ of degree $n_d$ which reduces to $a^{n_d}$ mod $p$.  Using
  \eqref{thm:witt-filtration-compat-with-p} we see that
  $\frac{I_{d+1}}{pI_d} = 
  \frac{I_{d+1}}{I_{d+1}\cap p\O} = \frac{I_{d+1}+p\O}{p\O}$, a cyclic
  $\O$-module generated by the image of $f_d$.  Thus
  $I_{d+1}=pI_d+(f_d)$.  
\end{proof}

\begin{exam}[Height 2 and $p=2$]
  It is possible to compute the polynomials  $f_d$ by hand with  effort.
  For the particular example discussed in 
  \cite{rezk-dyer-lashof-example}, one can  
  show that
  \begin{align*}
    \O/I_1&=\Z_2\powser{a}/(2,a),
    \\
    \O/I_2 &=\Z_2\powser{a}/(4,2a, a^3-2),
    \\
    \O/I_3& =\Z_2\powser{a}/(8,4a,2a^3-4, a^7-2a).
  \end{align*}
\end{exam}

\section{Classical Witt vectors}
\label{sec:classical-witt}

In the following sections we construct and compute the splitting of
$\Witt_p$ from $\Witt_E$.  We begin by collecting some facts about the
classical Witt vectors (integral 
$\Witt_\Z$ and $p$-typical $\Witt_p$), whose proofs are largely left
to the references.  Our presentation is a bit unconventional, as  it is meant
to make evident the analogy with the Lubin-Tate-Witt vectors.  but it
is not really new, being a manifestation of the fact that the 
theory of integral Witt vectors is subsumed in that of
$\lambda$-rings\footnote{Apparently first observed by Cartier around
  1968.}, while the theory of $p$-typical Witt vectors is similarly
related to  $\delta$-rings \cite{joyal-delta-rings}.

The only novel feature worth noting is the definition of higher height
analogues of the Artin-Hasse exponential, which will play the main role in
our proof of the splitting of $\Witt_E$.  

In the following, $RG$ is the complex representation ring of a finite
group $G$.  We recall that $RG$ is always a finitely generated free abelian
group, and that the functor is symmetric monoidal, e.g., $R(G\times H)\approx
RG\otimes RH$.  

\subsection{Integral Witt vectors}\label{subsec:integral-witt}

For a commutative ring $A$,  define $\Witt_\Z(A)$  as the subset
\[
\Witt_\Z(A) \subseteq \prod_{m\geq0} R\Sigma_m \otimes A
\]
consisting 
of $a=(a_m)$ such that $a_0=1$ and $\Res_{\Sigma_i\times
  \Sigma_j}^{\Sigma_{i+j}}a_{i+j}
= a_i\otimes a_j$.  Then $\Witt_\Z(A)$ is a commutative ring under
\[
(a+b)_m = \sum_{i+j=m} \Ind_{\Sigma_i\times \Sigma_j}^{\Sigma_m}
(a_i\otimes b_j), \qquad (ab)_m=a_mb_m.
\]
Then $\Witt_\Z$ is corepresented by the ring
\[
T_\Z \defeq \bigoplus_{m\geq0} R\Sigma_m^\vee,
\]
where $M^\vee=\Hom_{\Mod_\Z}(M,\Z)$ and 
multiplication is induced by restriction along $\Sigma_i\times
\Sigma_j\leq \Sigma_{i+j}$.  
We have \cite{atiyah-power-operations}*{Cor.\ 1.3} an isomorphism of rings
\[
\Z[s_m,\; m\geq 1] \xra{\sim} T_\Z,  
\]
where $s_m\colon R\Sigma_m\ra \Z$ is given on representations by  $V\mapsto
\dim (V^{\Sigma_m})$.  (Note that $s_0=1$.)

Recall that all characters of symmetric group representations are
integer valued.
Let $w_n\in R\Sigma_n^\vee$ denote the map $V\mapsto \chi(V,c_n)$,
where $c_n\in \Sigma_n$ is any cycle of length $n$.  The monomial
$w_{n_1}\cdots w_{n_k}\in R\Sigma_n^\vee$, $n=\sum n_i$, then corresponds
to $V\mapsto \chi(V,g)$, where $g\in \Sigma_{n_1+\cdots+n_k}$ is any
product of disjoint cycles of 
lengths $n_1,\dots,n_k$.   The formula $s_m(V) =
\frac{1}{m!}\sum_{g\in \Sigma_m} \chi(V,g)$ gives rise to the first
equality 
in 
\[
  \sum_{m\geq0} s_m\; X^m
  = \exp\biggl[ \sum_{n\geq 1} \frac{w_n}{n}\; X^n\biggl]
  = \prod_{k\geq1} (1-x_k\, X^k)^{-1} \in T_\Z\powser{X}.
\]
The second equality gives another  polynomial basis $\Z[x_k,\;
k\geq1] \approx T_\Z$, whose elements are related to the $w_n$s by the
Witt polynomials $w_n=\sum_{k\mid n} kx_k^{n/k}$.
Evaluation at the elements $s_m$ give an isomorphism of abelian groups
\[
s\colon  \Witt_\Z(A)
  \xra{\sim} (1+X\, A\powser{X})^\times, \qquad a\;\mapsto \;
  \sum_{m\geq0} s_m(a_m)\, X^m,
\]
while evaluation at  $w_n$ gives a ring homomorphism $\gh_n\colon
\Witt_\Z(A) \ra A$, the $n$th \dfn{ghost component}.

\subsection{$p$-typical Witt vectors}

Fix a prime $p$.  Define the \dfn{$p$-typical representation ring}
$R_pG$ of
a finite group $G$ to be the image of $\Res_P^G\colon RG\ra RP$, where
$P\leq G$ is any $p$-Sylow subgroup.   Thus $R_pG$ is a quotient of
$RG$, in which virtual representations are identified when their
characters coincide when evaluated at elements of $p$th power order.
Note that $R_pG$ is also a finitely generated free abelian group, that as a
functor on the category of groups $R_p$ is symmetric monoidal, and that induced
representations pass to $R_p$: that is, if $H\leq G$, we have a map
$\Ind_H^G \colon R_pH\ra R_pG$ compatible under the quotient with the
induced representation map $\Ind_H^G\colon RH\ra RG$.

For a commutative ring $A$, define $\Witt_p(A)$ as a subset
\[
  \Witt_p(A) \subseteq \prod_{m\geq0} R_p\Sigma_m\otimes A 
\]
analogously to our definition of $\Witt_\Z(A)$.  
Then $\Witt_p$ is
corepresented by the ring
\[
  T_p \defeq \bigoplus_{m\geq0} R_p\Sigma_m^\vee.
\]
Note that $T_p$ is a subring of $T_\Z$, which furthermore is
a summand of $T_\Z$ as an abelian group.  This inclusion induces
a natural ring homomorphism
$\Witt_\Z(A)\ra
\Witt_p(A)$.  Furthermore, we have 
that $w_{p^d}, x_{p^d}\in T_p$ for $d\geq0$, and an isomorphism
of rings
\[
  \Z[x_{p^d},\; d\geq0]\xra{\sim} T_p.
\]
This implies that $\Witt_\Z(A)\ra \Witt_p(A)$ is always surjective.

\subsection{$p$-typical ghost components}

Let $\gh_{p^d}\colon \Witt_p(A)\ra A$ denote the ring homomorphism
induced by evaluation at $w_{p^d}\in T_p$.

\begin{lemma}
For each $d\geq0$ and ring $A$ we have commutative diagrams of rings
\[\xymatrix{
    {\Witt_p(A)} \ar[r]^-{\gh_{p^d}} \ar@{->>}[d]_{\Witt_p(\pi)}
    & {A} \ar@{->>}[d]^{\pi}
    & {\Witt_p(A/p)} \ar[r]^{\wt{\gh}_{p^{d+1}}} \ar[d]_{\Witt_p(\phi)}
    & {A/p^{d+2}} \ar@{->>}[d]^{\pi}
    \\
    {\Witt_p(A/p)} \ar[r]_-{\wt{\gh}_{p^d}}
    & {A/p^{d+1}}
    & {\Witt_p(A/p)} \ar[r]_-{\wt{\gh}_{p^d}}
    & {A/p^{d+1}}
  }  \]
where $\phi\colon A/p\ra A/p$ is the $p$th power map, and $\pi$
denotes the evident quotient maps.
\end{lemma}
\begin{proof}
  For an element $x=(x_{p^k})_{k\geq0}\in \Witt_p(A)$, note that
  $\gh_{p^d}(x)$ is given by the Witt polynomial
  $\gh_{p^d}(x)=\sum_{j=0}^d p^j x_{p^j}^{p^{d-j}}$.  Thus if
  $x=(x_{p^j}), y=(y_{p^j})\in \Witt_p(A)$ are such that
  $x_{p^j}\equiv y_{p^j} \pmod {pA}$, an elementary argument gives
  \[
    x_{p^j}^{p^{d-j}} \equiv y_{p^j}^{p^{d-j}} \pmod{p^{d-j+1}A}, \quad
    \text{whence} \quad \gh_{p^d}(x)\equiv \gh_{p^d}(y) \pmod{p^{d+1}A}.  
  \]
  Thus the map $\wt{\gh}_{p^d}\colon \Witt_p(A/p) \ra A/p^{d+1}$ is
  defined and makes the left-hand square commute, and is a ring
  homomorphism since $\Witt_p(\pi)$ is surjective.  That the
  right-hand square commutes is an elementary calculation with the
  Witt polynomials.  
\end{proof}

\subsection{$p$-typical Witt vectors of perfect rings}
\label{subsec:p-typical-witt-of-perfect}

The maps $\wt{\gh}_{p^d}\colon \Witt(A/p)\ra A/p^{d+1}$  assemble
into a ring homomorphism
  \[
\lim_{\longleftarrow} \bigl( \Witt_p(A/p), \Witt_p(\phi) \bigr) \ra
\lim_{\longleftarrow} A/p^{d+1}. 
\]
When $A$ is $p$-complete and $\kappa=A/p$ is a perfect
$\F_p$-algebra, we get a 
ring homomorphism $\wt\gh_{p^\infty}\colon \Witt_p(\kappa)\ra A$, which
exhibits $\Witt_p(\kappa)$ as the initial example of a $p$-complete ring
$A$ equipped with an isomorphism $A/p\approx \kappa$.  In particular,
we obtain the following.

\begin{cor}\label{cor:witt-fp-is-p-padics}
  There is an isomorphism of rings $\wt\gh_{p^\infty}\colon
  \Witt_p(\F_p)\xra{\sim} \Z_p$ defined by
  $\wt\gh_{p^\infty}(a)=\lim_{d\to\infty} \gh_{p^d}(\wt{a})$, where
  $\wt{a}\in \Witt_p(\Z_p)$ is any lift of $a$.
\end{cor}

\subsection{Units in Witt vectors}

\begin{prop}\label{prop:units-in-witt-vectors}
  We have that
\begin{align*}
  \Witt_\Z(A)^\times
  &= \set{a\in \Witt_\Z(A)}{ \text{$\gh_m(a)\in
    A^\times$ for all $m\geq 1$}},
  \\
  \Witt_p(A)^\times
  &= \set{a\in \Witt_p(A)}{ \text{$\gh_{p^d}(a)\in
    A^\times$ for all $d\geq 0$}}.
\end{align*}
\end{prop}
\begin{proof}
  We write the proof in the case of $\Witt_\Z$, as the proof for
  $\Witt_p$ can be proved by exactly the same method.  Note that $\subseteq$ is
  automatic, since ghost components are ring homomorphisms.

  For this proof we freely use the fact that $\Witt_\Z$ is the comonad
  whose coalgebras are $\lambda$-rings.  In particular, $T_\Z$ is
  the free $\lambda$-ring on one generator $x$.   Consider $i\colon
  T_\Z \ra T_\Z[w^{-1}]$, where the target is the ring
  obtained by formally inverting all the $w_m$ for $m\geq 1$.   Then
  $i$ is a morphism of $\lambda$-rings by
  \cite{wilkerson-lambda-rings}*{Cor.\ 1.3}, which asserts that
  inverting a multiplicatively closed subset of a $\lambda$-ring which
  is closed under Adams operations yields a $\lambda$-ring.

Thus given $a\in \Witt_\Z(A)$ with $\gh_m(a)\in A^\times$, and so
represented by a ring map $a'\colon T_\Z[w^{-1}]\ra A$ such that
$a'\circ i=a$,  we have ring
homomorphisms  
  \[
  \Z\{x\}[w^{-1}] \xra{\nabla} \Witt_\Z (\Z\{x\}[w^{-1}])
  \xra{\Witt_\Z(a')} \Witt_\Z(A)
\]
where $\nabla$ is the coalgebra structure map defining the
$\lambda$-ring structure on $T_\Z[w^{-1}]$.  The image of $x$
under $\Witt_\Z(a')\circ \nabla$ is $a$, which is therefore invertible
being the image of the invertible element $x$.  
\end{proof}

\subsection{Height $h$ Artin-Hasse exponentials}

\begin{lemma}\label{lemma:gen-conj-number-is-p-local}
  Let $G$ be a finite group, $p$ a prime, and  $h\geq1$.  Then
  \[
    \frac{\len{\Hom(\Z_p^h, G)}}{\len{G}} \in \Z_{(p)}.
  \]
\end{lemma}
\begin{proof}
  Write $G_m=\set{g\in G}{g^m=1}$ and $G_{p^\infty}=\bigcup G_{p^k}$. 
  When $h=1$ the claim 
  is immediate from Frobenius' theorem\footnote{See
    \cite{speyer-counting-frobenius} for  
    a recent elegant proof.} that $\len{G_m}/m\in \Z$.  For the 
  general case use induction on $h$ and the evident identity
  \[
    \frac{\len{\Hom(\Z_p^h,G)}}{\len{G}}
    = \sum_{g\in G_{p^\infty}} \frac{\Hom(\Z_p^{h-1},
      \Cent_G(g))}{\len{G}} =
\sum_{[g]\in G_{p^\infty}^{\mr{cl}}} \frac{\len{\Hom(\Z_p^{h-1},
      \Cent_G(g))}}{\len{\Cent_G(g)}},
  \]
  the latter sum taken over representatives of conjugacy classes in
  $G_{p^\infty}$.  
\end{proof}

\begin{rem}
  Some readers may find it easier to observe that
  $\len{\Hom(\Z_p^h,G)}/\len{G}$ is the height $h$ chromatic
  cardinality of $BG$, and so necessarily a $p$-local integer
  \cite{benmoshe-higher-semiadditivity-transchromatic}*{Theorem B}. 
\end{rem}

Fix a prime $p$.  As is our convention  (\S\ref{subsec:conventions}) we write
\begin{align*}
  N^h_m &\defeq \len{\bigl\{\text{subgroups of order $m$ in
          $(\Q_p/\Z_p)^h$}\bigr\}},
\end{align*}
so that $N^h_m=0$ unless $m$ is a $p$th-power.  We note the
generating series
    \[
      N^h(T)\defeq \sum_{d\geq0} N^h_{p^d}\, T^d = \frac{1}{(1-T)(1-pT)\cdots
        (1-p^{h-1}T)} \in \Z\powser{T},
    \]
    whence $N^h_{p^d}\equiv 1 \pmod p$.
\begin{lemma}\label{lemma:Nhp-limit}
  For a prime $p$ and $h\geq 1$, the sequence $\{
  N^h_{p^d}\}_{d\geq0}$ converges $p$-adically: we have 
  \[
    N^h_{p^\infty} \defeq \lim_{d\to \infty} N^h_{p^d} =
    \frac{1}{(1-p)(1-p^2)\cdots (1-p^{h-1})} \in \Z_p.
    \]
  \end{lemma}
  \begin{proof}
    The generating series gives
 $N^h_{p^\infty} = (1-T)\,N^h(T) \bigm|_{T=1}$.
  \end{proof}

The following defines the \dfn{height $h$ Artin-Hasse exponential}
$\AH^h_\Z$ (at the prime $p$), which may be viewed as a Witt vector in the $p$-local
integers.    The height 1 case is precisely 
the classical Artin-Hasse exponential.  
\begin{lemma}\label{lemma:ht-h-artin-hasse}
  For $h\geq 1$ we have an identity
  \[
    \AH_\Z^h\defeq \sum_{m\geq0} \frac{\len{\Hom(\Z_p^h,
        \Sigma_m)}}{m!} X^m = \exp \biggl[\;\sum_{d\geq0} 
      \frac{N^h_{p^d}}{p^d}\, X^{p^d}\;\biggr] \in
      (1+X\,\Z_{(p)}\powser{X})^\times 
      \approx \Witt_\Z(\Z_{(p)}).
    \]
    In particular, $\gh_{p^d}(\AH^h_\Z)=N^h_{p^d}$, with
    $\gh_m(\AH^h_\Z)=0$ when $m$ is not a $p$th power.
  \end{lemma}
  \begin{proof}
    The identity is  a special case of Wohlfahrt's formula
    \[
      \sum_{m\geq0} \frac{\len{\Hom(G,\Sigma_m)}}{m!} X^m = \exp
      \biggl[\; \sum_{H\leq G} \frac{X^{\len{G:H}}}{\len{G:H}} \biggr],
    \]
    originally proved in \cite{wohlfahrt-satz-von-dey} for homomorphisms from
    finite $G$, but which 
    evidently applies with $G=\Z_p^h$ as well (see
    \eqref{rem:wohlfahrt-identity} below). 
    That the coefficients are $p$-local integers is
    \eqref{lemma:gen-conj-number-is-p-local}.   The identification of
    ghost components follows from the discussion in
    \eqref{subsec:integral-witt}. 
  \end{proof}

We write $\AH^h_p\in
\Witt_p(\Z_{(p)})$ for the image of $\AH^h_\Z$  under the projection
$\Witt_\Z(\Z_{(p)})\ra 
\Witt_p(\Z_{(p)})$.

\begin{prop}\label{prop:ht-h-artin-hasse-in-fp}
  For all $h\geq 1$ the element $\AH^h_p\in \Witt_p(\Z_{(p)})$ is
  a unit.  In particular, $\AH^1_p$ is the identity element of the
  ring 
  $\Witt_p(\Z_{(p)})$. 
Furthermore, the image of $\AH^h_p$ under the projection to 
$ \Witt_p(\F_p)=\Z_p$ is $N^h_{p^\infty}=[(1-p)\cdots
  (1-p^{h-1})]^{-1}$.  
\end{prop}
\begin{proof}
  From \eqref{lemma:ht-h-artin-hasse} we read off that
  $\gh_{p^d}(\AH_p^h)=N^h_{p^d}$.
  That $\AH_p^h$ is a unit is immediate from
  \eqref{prop:units-in-witt-vectors}, and we use
  \eqref{cor:witt-fp-is-p-padics} and 
  \eqref{lemma:Nhp-limit}   to  
  compute its image in $\Witt_p(\F_p)=\Z_p$. 
\end{proof}

Now we can prove  that, on $p$-local rings, $\Witt_p$ is a summand of
$\Witt_\Z$ as a $\Witt_\Z$-module.
\begin{prop}\label{prop:p-typical-witt-summand}
  Restricted to $\Z_{(p)}$-algebras, the projection $\pi\colon
  \Witt_\Z\ra \Witt_p$ admits a section, which we denote by $j$, which
  is a map of $\Witt_\Z$-modules, and  with
  the property that for all $h\geq1$ the diagram
  \[\xymatrix{
      {\Witt_\Z} \ar@{->>}[d]_{\pi} \ar[r]^{\AH^h_\Z\cdot}
      & {\Witt_\Z} 
      \\
      {\Witt_p} \ar[r]_{\AH^h_p\cdot}^{\approx}
      & {\Witt_p} \ar@{>->}[u]_{j}
    }\]
  commutes.  In particular, $\Witt_p$ is a summand of $\Witt_\Z$,
  split off by the idempotent $\AH_\Z^1$.  
\end{prop}
\begin{proof}
In the following we implicitly assume that all rings are
$p$-localized, so $T_\Z$ is really $T_\Z\otimes \Z_{(p)}$,
etc.  We write $\pi^*\colon T_p\ra T_\Z$ for the inclusion
representing $\pi$, and $[\AH^h_\Z]\colon T_\Z\ra T_\Z$ and
$[\AH^h_p]\colon T_p\ra T_p$ for the maps representing
multiplication by the Artin-Hasse elements.  
Recall that $T_\Z$ and $T_p$ are torsion free and generated
up to torsion by their Witt-polynomial  elements $w_n$, so to verify that
two ring homomorphisms between them (or tensor products of copies of
them) are equal it suffices to evaluate them at Witt-polynomial elements.

Pick any retraction $r\colon T_\Z\ra T_p$ of $\pi^*$ by a
ring homomorphism, and define $j^*\defeq [\AH^1_\Z]\circ r$.
Evaluation at ghost components shows (1) $j^*\circ \pi^*=\id_{T_p}$,
(2) $j^*$ does not depend on the choice of $r$, (3) $j^*$ induces a 
map $j\colon \Witt_p\ra \Witt_\Z$  of $\Witt_\Z$-module schemes, and (4) we have 
identities $[\AH^h_p]\circ j^*= \pi^*\circ [\AH^h_\Z]$.
\end{proof}

\section{The map $\beta\colon \Witt_\Z\ra \Witt_E$}
\label{sec:beta}

In this section we construct a natural ring map $\beta\colon
\Witt_\Z\ra \Witt_E$ and prove the following result on its effect on
ghost components.  Thanks to Nat Stapleton for pointing out how the
constructions of \cite{cdmmrs-total-power-burnside} can be used here.
\begin{prop}\label{prop:beta-on-ghost}
  For all $d\geq0$ and commutative $\O$-algebras $A$ there are
  commutative squares 
  \[\xymatrix{
    {\Witt_\Z(A)} \ar[r]^-{\beta} \ar[d]_{\gh_{p^d}}
    & {\Witt_E(A)} \ar[d]^{\gh_d}
    \\
    {A} \ar[r]_-{a\mapsto 1\otimes a}
    & {\O_d\otimes_\O A}
    }\]
\end{prop}

We write $A(G)$ for the Burnside ring of a finite group $G$.  Recall
that there is an evident ``linearization'' map $\ell\colon A(G)\ra
RG$ to the complex representation ring, sending $\ell(S)\defeq \C[S]$.

Following \cite{cdmmrs-total-power-burnside}, we 
say that a $\Sigma_m$-set is \dfn{submissive} if it admits an
equivariant embedding into $S^{\times m}$ (with evident $\Sigma_m$-action) for
some set $S$.  Let $\oA(m)\subseteq A(\Sigma_m)$ denote the subgroup
of the Burnside ring generated by the finite and transitive submissive
$\Sigma_m$-sets.  
\begin{prop}\label{prop:submissive-properties}\forcepar
  \begin{enumerate}
  \item The subgroup $\oA(m)\subseteq A(\Sigma_m)$ is a subring, and is
    exactly the Grothendieck ring of finite submissive
    $\Sigma_m$-sets.
  \item The restriction of $\ell$ to a map $\oA(m)\xra{\ell}
    R\Sigma_m$ 
    is an isomorphism of rings.
  \item   The transfer map $A(\Sigma_i)\otimes A(\Sigma_j)\ra
    A(\Sigma_{i+j})$ restricts to a map $\oA(i)\otimes \oA(j)\ra \oA(i+j)$.
  \item   The restriction map $A(\Sigma_{i+j})\ra A(\Sigma_i\times
    \Sigma_j)$ restricts to a map $\oA(i+j)\ra \oA(i)\otimes \oA(j)$.  
  \end{enumerate}
\end{prop}
\begin{proof}
  (1) is \cite{cdmmrs-total-power-burnside}*{Prop.\ 3.3}.  (2) is
  standard; see \cite{cdmmrs-total-power-burnside}*{\S6} and
  \eqref{rem:rep-symm-basis}.  (3) is 
  \cite{cdmmrs-total-power-burnside}*{Prop.\ 3.5}.

  Statement (4) is
  with respect to $\oA(i)\otimes \oA(j)\ra A(\Sigma_i)\otimes
  A(\Sigma_j)\ra A(\Sigma_{i+j})$, where the second map is the
  external product.  To prove it, it suffices to note that if
  $X\subseteq S^{\times i}\times S^{\times j}$ is a transitive $\Sigma_i\times
  \Sigma_j$-subset, then it is of the form $X=X'\times X''$ for 
  invariant subsets $X'\subseteq S^{\times i}$ and $X''\subseteq
  S^{\times j}$.
\end{proof}

\begin{rem}\label{rem:rep-symm-basis}
  Statement (2) of \eqref{prop:submissive-properties} amounts to the
  fact that, as a free abelian group $R\Sigma_m$ has a basis given by
  the representations $\rho_{\ul{m}}$ induced from the trivial
  representation of subgroups of the form $\Sigma_{m_1}\times \cdots
  \times \Sigma_{m_r}$ \cite{atiyah-power-operations}*{Prop.\ 1.1}.  
\end{rem}

\begin{cor}
  The isomorphism $\oA(m)\approx R\Sigma_m$ of rings of
  \eqref{prop:submissive-properties}(2) is 
  compatible with restriction and transfer along $\Sigma_i\times
  \Sigma_j\ra \Sigma_{i+j}$.  
\end{cor}

Next we recall the map $A(G)\ra \pi^0 \Sip BG$ from the
Burnside ring to the stable cohomotopy of $BG$.  Explicitly, this
sends a $G$-set $X$ to the composite $\Sip BG=\Sip (*_{hG})\ra \Sip
(X_{hG}) \ra \Sip(*)=\mb{S}$, where the first map in the sequence is
transfer, and the second is projection.

Let $E$ be a Morava $E$-theory, and let $b_m$ denote the composite of the
sequence
\[
R\Sigma_m \xla{\sim} \oA(m) \rightarrowtail A(\Sigma_m)\ra
\pi^0\Sip BG \ra E^0BG,
\]
where the last map is induced by the unit map $\mb{S}\ra E$.

\begin{prop}
The function $\beta(a)\defeq \bigr((b_m\otimes \id_A)(a_m)\bigr)$ defines a ring 
homomorphism $\Witt_\Z(A)\ra \Witt_E(A)$, natural in $\O$-algebras $A$.
\end{prop}
\begin{proof}
  The statement refers to functions  $b_m\otimes \id_A\colon
  R\Sigma_m\otimes A\ra E^0B\Sigma_m\otimes_\O A$, which are clearly
  natural in the $\O$-algebra $A$.  We need only to verify that they
  descend to a function $\Witt_\Z(A)\ra \Witt_E(A)$ and is a ring
  homomorphism, i.e., that they are compatible with restriction,
  transfer, and multiplication, which is immediate from
  \eqref{prop:submissive-properties}. 
\end{proof}

\begin{proof}[Proof of \eqref{prop:beta-on-ghost}]
  This is immediate from the the commutativity of the square of ring
  homomorphisms 
  \[\xymatrix{
    {R\Sigma_{p^d}} \ar[r]^{b_{p^d}} \ar[d]_{w_{p^d}}
    & {E^0B\Sigma_{p^d}} \ar[d]
    \\
    {\Z} \ar[r]
    & {E^0B\Sigma_{p^d}/\trJ_d=\O_d}
  }\]
As noted above \eqref{rem:rep-symm-basis}, $R\Sigma_m$ is a free
abelian group on representations 
$\rho_{\ul{m}}$ induced from $\Sigma_{m_1}\times\cdots\times
\Sigma_{m_r}$.  
  Thus the kernel of $w_m\colon
  R\Sigma_m\ra \Z$, which is evaluation of characters at a cycle of
  length $m$, is spanned by the $\rho_{\ul{m}}$ induced from proper
  subgroups.  
  Since $b_{p^d}$ preserves
  transfers, it sends $\ker w_{p^d}$ into the transfer ideal $\trJ_d$.
\end{proof}

\section{The map $\sigma\colon \Witt_E\ra \Witt_\Z$}
\label{sec:sigma}

In this section we describe a natural abelian group homomorphism
$\sigma\colon \Witt_E\ra \Witt_\Z$ and prove the following result on
its effect on ghost components.
Recall that the Strickland rings $\O_d$ are finitely generated free
$\O$-modules of rank $N^h_{p^d}$ \eqref{prop:o-d-free-rank}
\begin{prop}\label{prop:tau-on-ghost}
For all $d\geq0$ and $\O$-algebras $A$ there are commutative squares
  \[\xymatrix{
      {\Witt_E(A)} \ar[r]^-{\sigma} \ar[d]_{\gh_d}
      & {\Witt_\Z(A)} \ar[d]^{\gh_{p^d}}
      \\
      {\O_d\otimes_\O A} \ar[r]_-{\tr}
      & {A}
    }\]
where $\tr(x)$ is the $\O$-linear trace of left multiplication by $x$.
Furthermore we have $\gh_m\circ \sigma=0$ when $m$ is not a $p$th power.
\end{prop}
The map $\sigma$ was originally  constructed in
\cite{bssw-rational-kn}*{\S2.5}.

Let $\Tr^e_{\sigma_m}\colon E^0B\Sigma_m\ra \O$ denote the
$K(n)$-local transfer along $\Sigma_m\ra e$ (defined due to the
1-semiadditivity of $\Sp_{K(h)}$).  
Recall the isomorphism $s\colon
\Witt_\Z(A)\xra{\sim}(1+X\,A\powser{X})^\times$ of abelian groups
\eqref{subsec:integral-witt}. 
\begin{prop}
  There is a natural abelian group homomorphism
  \[
  \sigma\colon \Witt_E(A) \ra \Witt_\Z(A)  
\]
characterized by the formula
\[
s(\sigma(a)) = \sum_{m\geq0} \Tr^e_{\Sigma_m}(a_m)X^m.
\]
\end{prop}
\begin{proof}
Elements $a\in \Witt_E(A)$ satisfy $a_0=1$, so the right-hand side
lies in the series with constant term 1. 
That it is a homomorphism of abelian groups follows from functoriality
of the transfer:
\begin{align*}
  \Tr_{\Sigma_{i+j}}^e \Tr_{\Sigma_i\times \Sigma_j}^{\Sigma_{i+j}}(a_i\otimes
  a_j) = \Tr^{e}_{\Sigma_i\times \Sigma_j}(a_i\otimes a_j) =
  \Tr_{\Sigma_i}^e(a_i)\Tr_{\Sigma_j}^e(a_j). 
\end{align*}
\end{proof}

\subsection{Exponential formula for marks}

Say that a group $G$ is \dfn{admissible} if it has only finitely many
subgroups of index $m$ for any $m\geq 1$.  For our purposes, the only
admissible groups we need to consider are $G=\Z$ and $G=\Z_p^h$.

A \dfn{mark} from an admissible group $G$ to a ring $R$ is a function $f$ which
sends a finite $G$-set $S$ to an element $f(S)\in R$, such that
\begin{enumerate}
\item [(i)] $f$ is isomorphism invariant: $S\approx T$ as $G$-sets
  implies $f(S)=f(T)$;
  \item [(ii)] $f$ is multiplicative: $f(\varnothing)=1$ and
    $f(S\amalg T)=f(S)f(T)$.  
  \end{enumerate}
When $\gamma\colon G\ra \Sigma_m$ is
homomorphism to a symmetric group, we abuse notation and write
$f(\gamma)$ for the value  
of $f$ on the evident $\Sigma_m$-set.  We will need the following
exponential identity for marks.
\begin{prop}\label{prop:exponential-identity-marks}
  For any mark $f$ on $G$ with values in a torsion free ring $R$, we have
  \[
    \sum_{m\geq0} \frac{X^m}{m!}\; \sum_{\gamma\colon G\ra \Sigma_m}
    f(\gamma) = \exp\left[ \sum_{H\leq G} \frac{X^{[G:H]}}{[G:H]} f(G/H)\right],
  \]
  where the sum on the right is taken over all subgroups of finite index.
\end{prop}
\begin{proof}
We can rewrite both sides of the identity in terms of isomorphism classes
of finite $G$-sets.  That is, it suffices to show
\[
  \sum_{[S]} \frac{X^{\len{S}}}{|\Aut(S)|}\, f(S) = \prod_{\text{tr.\
      $[T]$}} \exp\left[ \frac{X^{|T|}}{|\Aut(T)|}\, f(T) \right]
\]
where the sum is over all isomorphism classes of finite $G$-sets,
while the product is over all isomorphism classes of \emph{transitive}
finite $G$-sets.  The equality then follows using the observation that
if $S\approx T_1^{\amalg k_1}\amalg\cdots \amalg T_r^{\amalg k_r}$
where the $T_j$ are pairwise-non-isomorphic transitive $G$-sets, then
$\Aut(S)\approx\prod_{j=1}^r \Aut(T_j)^{k_j}\rtimes \Sigma_{k_j}$.  
\end{proof}
\begin{rem}\label{rem:wohlfahrt-identity}
  When $f\equiv 1$ this specializes to  Wohlfarht's identity.
\end{rem}

\subsection{Hopkins-Kuhn-Ravenel generalized character theory}

We recall basic facts about the HKR generalized character map
\cite{hopkins-kuhn-ravenel}, which we formulate 
as follows.  

There exists a
faithfully flat 
$p^{-1}\O$-algebra $C$  and for each homomorphism $\gamma\colon
\Z_p^h\ra G$ an $\O$-algebra map
\[
\chi_G(\gamma)\colon   E^0BG \ra E^0BG\otimes_\O C.
\]
The maps $\chi_G(\gamma)$ only depend on the conjugacy class of
$\gamma$, and they give rise to an isomorphism of $C$-algebras
\[
\chi_G\colon E^0BG\otimes_\O C \xra{\sim} \prod_{\Hom(\Z_p^h,G)_{\cnj}}
C,
\]
the product being taken over conjugacy classes of homomorphisms.

For our purposes, it is useful to formulate this in terms of base-change to an
arbitrary $C$-algebra.  Thus, we regard character theory as giving,
naturally in $C$-algebras $A$,   an isomorphism of $A$-algebras
of the form
\[
  \chi_G\colon E^0BG\otimes_\O A\xra{\sim} \prod_{\Hom(\Z_p^h,G)_\cnj}
  A,
\]
and we write $\chi_G(\gamma)\colon E^0BG\otimes_\O A\ra A$ for its
component at $\gamma\colon \Z_p^h\ra G$.

We note the formula for ``induced characters'', associated to
transfers along general homomorphisms (which exist because of the
1-semiadditivity of the $K(n)$-local category).
\begin{prop}\label{prop:induced-character}
Let $\phi\colon H\ra G$ be a homomorphism
  of finite groups, and write $\Tr_\phi\colon E^0BH\otimes_\O A\ra
  E^0BG\otimes_\O A$ for the
  induced transfer.  When $A$ is a $C$-algebra, we have for
  $\gamma\colon \Z_p^h\ra G$ the identity
  \[
  \chi_G(\gamma)\circ \Tr_\phi = \frac{1}{\len{H}}
  \sum_{\substack{\delta\colon \Z_p^h\ra H \\ u\in G,\; \phi\delta
      =u^{-1}\gamma u}} \chi_H(\delta).
\]
\end{prop}
This should be  well-known.  For instance, it can be read off from a
special case of the main theorem of
\cite{benmoshe-higher-semiadditivity-transchromatic}.  In fact, we 
only need  two special cases: (1) when $\phi$ is
injective \cite{hopkins-kuhn-ravenel}*{Theorem D}; (2) when $G$ is trivial
\cite{ganter-orb-genera}*{Prop.\ 7.9}.

\subsection{Generalized characters for symmetric groups}

As we are mainly  interested in the case of $G=\Sigma_m$,  we write
$\chi_m$ for $\chi_{\Sigma_m}$.  Note that conjugacy
classes of homomorphisms $\gamma\colon \Z_p^h\ra \Sigma_m$ correspond
exactly to isomorphism classes finite sets $S$ of order $m$ equipped with a
$\Z_p^h$-action. Thus in this case we abuse 
notation, and write $\chi_m(S)$ for $\chi_m(\gamma)$.
In particular, a subgroup $H\leq \Z_p^h$ of index $p^d$ determines a
$\Z_p^h$-set $\Z_p^h/H$, and thus a conjugacy class of homomorphisms
$\Z_p^h\ra \Sigma_{p^d}$  with
associated character map $\chi_{p^d}(\Z_p^h/H)$.

\begin{prop}\label{prop:character-formula-for-tau}
  Let $A$ be a $C$-algebra.  Then for $a\in \Witt_E(A)$ we have
  \[
  s(\sigma(a)) =
  \sum_{m\geq0} \frac{X^m}{m!} \sum_{\gamma\colon \Z_p^h\ra \Sigma_m}
  \chi_m(\gamma)(a_m) = 
\exp\biggl[ \sum_{d\geq0} \frac{X^{p^d}}{p^d}
  \sum_{\substack{ H\leq \Z_p^h \\ [\Z_p^h:H]=p^d}}
    \chi_{p^d}(\Z_p^h/H)(a_{p^d}) \biggr].
  \]
\end{prop}
\begin{proof}
  The first equality is immediate from the induced character formula
  \eqref{prop:induced-character}.  To prove the
  second equality, note that the function $f(\gamma)\defeq
  \chi_m(\gamma)(a_m)$ defines a mark on $\Z_p^h$ with values in $A$,
  so the exponential identity \eqref{prop:exponential-identity-marks}
  applies. 
\end{proof}

\begin{prop}\label{prop:character-ghost}
  For $C$-algebras $A$ there is a commutative square of $C$-algebra
  homomorphisms 
\[\xymatrix@C=50pt{ 
    {E^0B\Sigma_{p^d}\otimes_{\O} A} \ar[r] \ar[d]_{\chi_{p^d}}^{\approx}
    & {\O_d\otimes_\O A} \ar[d]^{\ol\chi_{p^d}}_{\approx}
    \\
    {\hspace{-1em} \prod_{\Hom(\Z_p^h,\Sigma_{p^d})_\cnj} \hspace{-2em} A} \ar[r]
  & {\prod_{\substack{H\leq \Z_p^h \\ [\Z_p^h:H]=p^d}}
    \hspace{-1em} A} 
}\]
in which the bottom horizontal map is the evident projection, and the
vertical maps are isomorphisms.
\end{prop}
\begin{proof}
The map $\chi_{p^d}$ is the character map, and so is an isomorphism.  

If $\gamma\colon \Z_p^h\ra \Sigma_m$ describes a non-transitive action
on $\{1,\dots,m\}$, then $\gamma(\Z_p^h)\leq \Sigma_i\times
\Sigma_{m-i}$ for some $0<i<m$.  Thus by formula for transfers from
subgroups \eqref{prop:induced-character}, we see that $\chi_{p^d}(\gamma)$ 
vanishes on the transfer ideal $\trJ_d$ for such $\gamma$, and thus
the map $\ol\chi_{p^d}$ exists.
\end{proof}

\begin{lemma}\label{lemma:tau-on-character}
  For $C$-algebras $A$ we have a commutative diagram
  \[\xymatrix@R=35pt{
    {\Witt_E(A)} \ar[r]^-{\gh_d} \ar[d]_{\sigma}
    & {\O_d\otimes_\O A} \ar[r]^-{\ol\chi_{p^d}}_-{\approx} \ar[d]^-{\tr}
    & {\prod_{\substack{H\leq \Z_p^h \\ [\Z_p^h:H]=p^d}} \hspace{-1em}
      A} \ar[d]^-{\text{sum}}
      \\
      {\Witt_\Z(A)} \ar[r]_-{\gh_{p^d}}
      & {A} \ar@{=}[r]
      & {A}
  }\]
\end{lemma}
\begin{proof}
  Let  $a\in \Witt_E(A)$.  Then we have
  \[
  \ol{\chi}_{p^d}(\gh_d(a)) = \sum_{\substack{H\leq \Z_p^h\\
      [\Z_p^h:H]=p^d}} \chi_{p^d}(\Z_p^h/H)(a_{p^d}) = \gh_{p^d}(\sigma(a)),
\]
where the first equality is \eqref{prop:character-ghost}, while
we can read off the second equality from 
\eqref{prop:character-formula-for-tau}.  
\end{proof}

We can now prove the main result of this section.
\begin{proof}[Proof of \eqref{prop:tau-on-ghost}]
  The statement is immediate for $p$-torsion free $A$ by
  \eqref{lemma:tau-on-character} and the fact that  $p^{-1}\O\ra C$ is
  faithfully flat.  The general statement follows by reduction to the
  universal example (the tautological element of
  $\Witt_E(\mc{O}\{x\})$).
\end{proof}

\begin{rem}
  The ghost components applied to $\sigma$ encode the action of
  ``Hecke operators''. 
  Suppose $A$ is a 
  $\mb{T}$-algebra (e.g., $A=\pi_0R$ for some $K(n)$-local
  commutative $E$-algebra), and write $P\colon A\ra
  \Witt_E(A)\subseteq \prod E^0B\Sigma_m\otimes_\O A$ for its total
  power operation.  Then
  \[
  \gh_d( \sigma(P(a))) = p^dT(p^d)(a),
\]
where $T(p^d)\colon A\ra p^{-1}A$ is  as in
\cite{rezk-units-and-logs}*{\S1.12}.  The components $\sigma_n$ of
$\sigma(x)= \sum_{n\geq0} \sigma_n(x)\, X^n$ are the symmetric power
operations introduced by Ganter \cite{ganter-orb-genera}.
\end{rem}

\section{The splitting of $\Witt_E$}
\label{sec:splitting}

We can now split $\Witt_p$ off of $\Witt_E$.  In what follows, we use
$\Witt_p$ and $\Witt_\Z$ to also denote the restriction of those
functors to $\O$-algebras.
\begin{thm}\label{thm:witt-splitting}
  There exist unique natural maps $\beta_p\colon \Witt_p\ra \Witt_E$ and
  $\sigma_p\colon \Witt_E\ra \Witt_p$ fitting into the commutative square
  \[\xymatrix{
    {\Witt_\Z} \ar[r]^-{\beta} \ar@{->>}[d]_{\pi} \ar@/^2pc/[rr]^{\AH^h_\Z\cdot}
    & {\Witt_E} \ar[r]^-{\sigma}  \ar@{>->}[dr]_{\sigma_p}
    & {\Witt_\Z} 
    \\
    {\Witt_p} \ar@{->>}[ur]_{\beta_p} \ar[rr]_{\AH^h_p\cdot}^{\approx}
    && {\Witt_p} \ar@{>->}[u]_{j} 
  }\]
of functors $\CAlg_\O\ra \Ab$.  Furthermore, $\beta_p$ is a natural
map of rings, and $\sigma_p$ is a natural map of $\Witt_p$-modules.
\end{thm}
\begin{proof}
For each of the functors $W\in \{\Witt_\Z, \Witt_p, \Witt_E\}$, their
collections of ghost components give injective homomorphisms when
evaluated on torsion free $\O$-algebras.  Since each of the $W$ is
corepresented by a torsion free $\O$-algebra, we can deduce any
identity of natural maps between them by proving they have the
same effects on ghost components.

Thus, by \eqref{prop:beta-on-ghost} and \eqref{prop:tau-on-ghost} we have that
$\gh_{p^d}(\sigma\beta(a))=N^h_{p^d}\gh_{p^d}(a)$, and $\gh_m(\sigma\beta(a))=0$
if $m$ is not a $p$th power.  Therefore $\sigma\beta(a)=\AH^h_\Z a$ using
\eqref{lemma:ht-h-artin-hasse}.

The same idea shows that both composites
\[
\Witt_\Z \xra{e\cdot} \Witt_\Z \xra{\beta} \Witt_E, \qquad
\Witt_E \xra{\sigma} \Witt_\Z \xra{e\cdot} \Witt_\Z
\]
are equal to $0$, where $e=1-\AH^1_\Z$.  The existence of $\beta_p$ and
$\sigma_p$ follows, as   
\eqref{prop:p-typical-witt-summand} implies that $\Im \pi= e\Witt_\Z$
and $\Ker j=   
\operatorname{Ann}(e)$.  This implies  $\sigma_p\beta_p(a)=\AH_p^h\cdot
a$, whence $\sigma_p\beta_p$ is a bijection since $\AH^h_p\in
\Witt_p(\Z_{(p)})^\times$. 
\end{proof}

\begin{prop}\label{prop:witt-filtration-restriction}
The Witt filtration of $\Witt_E$ restricts to that of $\Witt_p$: i.e.,
$\Witt_p(A)\cap I_d(A)$ is the $d$th Witt filtration ideal in
$\Witt_p(A)$.  
\end{prop}
\begin{proof}
Both $\beta_p$ and $\sigma_p$ preserve the relevant Witt filtrations, as
does multiplication by $\AH^h_p$ on $\Witt_p$.  The claim follows.
\end{proof}

We can explicitly compute the composite of $\sigma_p$ and $\beta_p$ in 
characteristic $p$.
\begin{prop}\label{prop:where-does-one-go}
If $A$ is an $\O$-algebra of characteristic $p$, then
$\sigma_p\beta_p\colon \Witt_p(A)\ra \Witt_p(A)$ is multiplication by the unit
$N^h_{p^\infty} = [(1-p)\cdots (1-p^{h-1})]^{-1}$.  
\end{prop}
\begin{proof}
  We just need to compute the image of $\AH_p^h$ under the composite of
  \[
  \Witt_p(\Z_{(p)}) \ra \Witt_p(\F_p) \ra \Witt_p(A),
\]
so the claim follows from \eqref{prop:ht-h-artin-hasse-in-fp}.
\end{proof}

%%% bibliography

\end{document}